\documentclass[a4paper]{amsart}
\usepackage{amsmath, amsthm, amssymb, amsfonts, amsthm, amscd, hyperref, mathrsfs,relsize}
\usepackage[arrow, matrix, curve]{xy}

\def\IQ{{\mathbb Q}}

\def\IZ{{\mathbb Z}}
\def\IN{{\mathbb N}}

\def\IL{{\mathbb L}}
\def\IH{{\mathbb H}}

\def\id{{\operatorname{id}}}
\def\Vect{{\operatorname{Vect}}}
\def\Mod{{\operatorname{Mod}}}
\def\aut{{\operatorname{aut}}}
\def\map{{\operatorname{map}}}
\def\Aut{{\operatorname{Aut}}}
\def\hom{{\operatorname{Hom}}}
\def\diff{{\operatorname{Diff}}}
\def\bdiff{{\widetilde{\operatorname{Diff}}}}
\def\baut{{\widetilde{\operatorname{aut}}}}
\def\Der{{\operatorname{Der}}}
\def\der{{\operatorname{Der}}}

\def\Derpw{{\operatorname{Der}^+_\omega}}
\def\univ{{\langle 1 \rangle}}

\def\interior{{\operatorname{int}}}
\def\g{{\mathfrak{g}}}
\def\Lie{{\mathscr{L}ie}}

\newtheorem{thm}{Theorem}[section]
\newtheorem*{THMA}{Theorem A}
\newtheorem*{THMB}{Theorem B}

\newtheorem{cor}[thm]{Corollary}
\newtheorem{prop}[thm]{Proposition}
\newtheorem{lem}[thm]{Lemma}

\theoremstyle{definition}
\newtheorem{defi}[thm]{Definition}
\theoremstyle{remark}
\newtheorem{rem}[thm]{Remark}

\newenvironment{example}[1][Example]{\begin{trivlist}
\item[\hskip \labelsep {\bfseries #1}]}{\end{trivlist}}

\begin{document}

\title{Rational Homological Stability for Automorphisms of Manifolds}
\author{Matthias Grey}
\date{\today}

\begin{abstract}
We show rational homological stability for the homotopy automorphisms and block diffeomorphims of iterated connected sums of products of spheres. The spheres can have different dimension, but need to satisfy a certain connectivity assumption.
The main theorems of this paper extend homological stability results for automorphism spaces of connected sums of products of spheres of the same dimension by Berglund and Madsen.
\end{abstract}

\maketitle
\setcounter{tocdepth}{1}
\tableofcontents

\section{Introduction}
Let $M$ be compact manifold with boundary with a chosen basepoint on the boundary. Denote by $\map_*(M,M)$ the space of pointed maps with the compact open topology. Denote by $\map_\partial(M,M)$ the subspace of maps, fixing the boundary point-wise. We denote by $\aut_\partial(M)$ the sub-monoid of homotopy self-equivalences fixing the boundary point-wise.

The first object we study in this paper is the rational homology of the classifying space $B\aut_\partial(M).$ More precisely, we show that it satisfies homological stability for certain families of manifolds. For example, for the $g$-fold connected sum $$N_g=\big (\mathlarger{\#}_g (S^p\times S^q) \big ) \smallsetminus \interior (D^{p+q})\text{, where }3\leq p<q<2p-1,$$ we show that $H_i(B\aut_\partial(N_g);\IQ)$ is independent of $g$ for $g>2i+2.$ Moreover we show that $H_i(B\aut_\partial(M\#N_g);\IQ)$ is independent of $g$ in the same range, where $M$ is some other connected sum of products of spheres. To make this statement precise, we introduce the following notation. 
Let $I$ be a finite indexing set for pairs of positive natural numbers $p_i,q_i$ and $n$ be a positive natural number, such that $3\leq p_i\leq q_i<2p_i-1$ and $p_i+q_i=n$ for all $i\in I$. Note that this implies that necessarily $n\geq 6.$ We define a smooth $n$-dimensional manifold with boundary diffeomorphic to $S^{n-1}$ $$N_I=\left(\mathlarger{\#}_{i\in I} (S^{p_i}\times S^{q_i})\right)\smallsetminus \interior (D^n)$$ and we assume a basepoint chosen on the boundary.

For a given $p\in \IN$ we define the ``generalized genus'' as $$g_p=\begin{cases} \operatorname{rank}(H_p(N_I))/2 &\text{ if } 2p=n \\ \operatorname{rank}(H_p(N_I)) &\text{ otherwise,}  \end{cases}$$ i.e. $g_p$ is the number of $S^p\times S^q$ summands of $N_I,$ where $q=n-p.$ In other words we see $N_I$ as the connected sum of the manifolds
$$N_I= \big(\mathlarger{\#}_{i\in \{j\in I|p_j\neq p  \}} S^{p_i}\times S^{q_i} \big) \# (\mathlarger{\#}_{g_p}S^{p}\times S^{q} \big) \smallsetminus \interior{D^{n}}.$$

Denote by $$V^{p,q}=S^p\times S^q\smallsetminus \interior (D_1^n\sqcup D_2^n), \text{ where $3\leq p\leq q<2p-1$ and $p+q=n.$}$$ We define a new manifold 
$$N'=N_I\cup_{\partial_1} V^{p,q},$$ by identifying one boundary component of $V^{p,q}$ with $\partial N_I.$ Note that $N'$ is canonically diffeomorphic to the manifold $N_{I'}$ with $I'=I\cup \{i'\}$, where $p_{i'}=p$ and  $q_{i'}=q.$ Using this, we define the stabilization map $$\sigma:\aut_\partial( N_I) \rightarrow \aut_\partial( N_I\cup_{\partial_1} V^{p,q})\xrightarrow{\cong} \aut_\partial( N_{I'}),  $$ by extending a self-map of $N_I$ by the identity on $V^{p,q}.$ In this paper we study the induced map on homology of the classifying spaces
$$\sigma_*:H_*(B\aut_\partial (N_I);\IQ)\rightarrow H_*(B\aut_\partial (N_{I'});\IQ).$$ The first main theorem of this paper is:

\begin{THMA}
The map
$$H_i(B\aut_\partial(N_I);\IQ)\rightarrow H_i(B\aut_\partial(N_{I'});\IQ)$$ 
induced by the stabilization map with respect to $V^{p,q},$ where $3\leq p\leq q < 2p-1$, is an isomorphism for $g_p>2i+2$ when $2p\neq n$ and $g_p>2i+4$ if $2p=n$ and an epimorphism for $g_p\geq 2i+2$ respectively $g_p\geq 2i+4$.
\end{THMA}

The block diffeomorphisms $\bdiff_\partial (X)$ are the realization of the $\Delta$-group, i.e. a simplicial group without degeneracies, with $k$-simplices, face preserving diffeomorphisms $$\varphi:\Delta^k\times X\rightarrow \Delta^k\times X,$$ such that $\varphi$ is the identity on a neighborhood of $\Delta^k\times \partial X.$ We map a $k$-simplex $\varphi$ to the $k$-simplex in $\bdiff_\partial(N_{I'})$ $$\Delta^k\times (N_I\cup_{\partial_1} V^{p,q})\rightarrow \Delta^k\times (N_I\cup_{\partial_1} V^{p,q}),$$ given by $\varphi$ on $\Delta^k\times N_I$ and the identity on $\Delta^k\times V^{p,q}$, where we use that $N_{I'}\cong N_I\cup_{\partial_1} V^{p,q}.$ This induces the stabilization map $B\bdiff_\partial(N_I)\rightarrow B\bdiff_\partial(N_{I'}).$

\begin{THMB}
The map
$$H_i(B\bdiff_\partial(N_I);\IQ)\rightarrow H_i(B\bdiff_\partial(N_{I'});\IQ)$$
induced by the stabilization map with respect to $V^{p,q},$ where $3\leq p\leq q < 2p-1,$ is an isomorphisms for $g_p>2i+2$ when $2p\neq n$ and $g_p>2i+4$ if $2p=n$ and an epimorphism for $g_p\geq 2i+2$ respectively $g_p\geq 2i+4$.
\end{THMB}

The argument is based on \cite{berg12,berg13}, where homological stability for the automorphism spaces of the manifolds $\mathlarger{\#_g}(S^d\times S^d) \smallsetminus \interior(D^{2d})$ is shown. Berglund and Madsen moreover determine the stable homology of the homotopy automorphisms and the block diffeomorphisms. The argument to determine the stable homology relies on \cite{galatiusrandalwilliams} where the stable cohomology of the moduli space of many even-dimensional smooth manifolds is determined. There is not an analogue of these results for odd-dimensional manifolds yet, but the homological stability results \cite{galatiusrandalwilliamshomstab} have been generalized to odd-dimensional manifolds \cite{perlmutter,perlmutter2}.
\subsection*{Outline of the argument}
The main idea of the proof is to consider the universal covering $$B\aut_\partial (N_I)\langle 1 \rangle \rightarrow B\aut_\partial (N_I)$$ or rather, the universal covering spectral sequence with $E^2$-page
$$H_s(\pi_1(B\aut_\partial (N_I));H_t(B\aut_\partial (N_I)\langle 1 \rangle;\IQ))\Rightarrow H_{s+t}(B\aut_\partial (N_I);\IQ).$$
If we can show that the stabilization map induces isomorphisms for large generalized genus $g_p$ in a range of $s$ and $t$, the spectral sequence comparison theorem implies homological stability for the homotopy automorphisms. So we reduced the problem to showing homological stability for the groups $\pi_1(B\aut_\partial (N_I))$ with certain twisted coefficients.

The first step is to determine the homotopy classes of homotopy automorphisms (Section \ref{on mapping class groups}), or rather the quotient acting non-trivially on the homology of the universal covering. More precisely, we determine the image and kernel of the ``reduced homology'' map $$\tilde H:\pi_0(\aut_\partial(N_I))\rightarrow \Aut (\tilde H_*(N_I))$$ to the automorphisms of the reduced homology as a graded group. For the manifolds $N_I$ the kernel is finite. This is in fact a consequence of the connectivity assumption $3\leq p_i\leq q_i < 2p_i-1$. The image which we call $\Gamma_I\subset \Aut (\tilde H_*(N_I))$ is the subgroup respecting the intersection form and when the dimension of $N_I$ is even, a certain tangential invariant. In particular we show later that elements in the kernel of $\tilde H_*$ act trivially on $H_t(B\aut_\partial (N_I)\langle 1 \rangle;\IQ).$ Thus we only have to study the groups $\Gamma_I.$

In Section \ref{automorphisms of hyperbolic modules} we review hyperbolic modules and give a slight generalization in order to describe the $\Gamma_I$ as automorphisms of an object with underlying graded $\IZ$-module $\tilde H_*(N_I)$.

In Section \ref{van der kallens and charneys homological stab results} we use homological stability results by Van der Kallen and Charney to show a homological stability result for the groups $\Gamma_I$ with certain twisted coefficient systems of ``finite degree''.
In Section \ref{On the rational homotopy type of homotopy automorphisms} we review results of Berglund and Madsen in order to show that the $H_t(B\aut_\partial (N_I)\langle 1 \rangle;\IQ)$ are in fact a twisted coefficient system satisfying homological stability for the $\Gamma_I.$ A crucial tool for showing twisted homological stability are Schur multifunctors (defined in Section \ref{Polynomial functors and Schur multifunctors}), which we use to determine the degree of the coefficient systems. In fact developing the Schur multifunctors as a tool to handle the different degrees coming from the homology $\tilde H_*(N_I)$ was one the the main hurdles in generalizing Berglund and Madsen's result.

To show the homological stability for the block diffeomorphisms, we consider the homotopy fibration $$\aut_\partial (N_I)/\bdiff_\partial (N_I) \to B\bdiff_\partial (N_I)\to B\aut_\partial (N_I).$$ We can determine the homology of a component of the homotopy fiber using surgery theoretic methods applied by Berglund and Madsen, which suffices to show homological stability using similar arguments as for the homotopy automorphisms.

\subsection*{Acknowledgments}
This paper is based on the authors PhD thesis, which was supported by the Danish National Research Foundation through the Centre for Symmetry and Deformation (DNRF92).The author wants to thank his advisor Ib Madsen for suggesting this project and his advice throughout working on it. The author moreover wants to thank Alexander Berglund for comments on the presentation of this article.

\section{Polynomial functors and Schur multifunctors}\label{Polynomial functors and Schur multifunctors}

We start by recalling the definition of polynomial functors in the sense of \cite[Section 9]{eilenberg} or rather slightly modified as in \cite[Section 3]{dwyer}. 
Let $T:\mathcal{A}\rightarrow \mathcal{B}$ be a (not necessarily additive) functor between abelian categories. The first cross-effect functor is defined to be $$T^1(X)=\operatorname{kernel}(T(X)\to T(0)),$$ where $X\rightarrow 0$ is the natural map to the zero object in $\mathcal{A}$. For $k>1$, the $k$-th cross-effect functor $$T^k:\mathcal{A}^k\rightarrow \mathcal{B},$$ is uniquely defined up to isomorphism given $T^l$ for $l<k$ by the properties:
\begin{enumerate} 
 \item $T^k(A_1,...,A_k)=T(0)$ if $A_i=0$ for some $i$.
 \item There is a natural isomorphism $$T(A_1\oplus ...\oplus A_k)\cong T(0)\oplus \underset{\{i_1,...,i_r\}}\bigoplus T^r(A_{i_1},...,A_{i_r}),$$ where the sum runs over all non-empty subsets $\{ i_1,...,i_r \}\subset\{ 1,...,k \}$.
\end{enumerate}

\begin{defi}
A functor $T$ is polynomial of degree $\leq k$, if $T^{l}$ is the constant zero functor for $l>k$. 
\end{defi}

An immediate consequence is that a functor is of degree $\leq 0$ if and only if it is constant. The higher cross-effects can be defined using deviations. The $k$-fold deviation of $k$ maps $f_1,...,f_k:A\rightarrow B$ in $\mathcal{A}$ is the map
$$T(f_1\top ... \top f_k):T(A)\rightarrow T(B),$$ given by $$T(f_1\top ... \top f_k)= T(0)\oplus \underset{\{ i_1,...,i_r \}}\bigoplus (-1)^{k-r} T(f_{i_1},...,f_{i_r} ) ,$$ where the sum runs over all non-empty subsets $\{ i_1,...,i_r \}\subset\{ 1,...,k \}$ and $0$ denotes the canonical map $A\rightarrow 0 \rightarrow B.$ Setting $A=A_1\oplus...\oplus A_k$ and denoting by $\pi_i:A\rightarrow A_i$ the projections and by $\iota_i:A_i\rightarrow A$ the inclusions, the $k$-th cross-effect functor is given on objects by $$T^k(A_1,...,A_k)=\operatorname{Image}(T(\iota_1\circ\pi_1\top ... \top \iota_k\circ\pi_k)).$$

The following properties of Polynomial functors are direct consequences of the definition.

\begin{prop}[See e.g. \cite{dwyer}]\mbox{}
\begin{enumerate}
 \item An additive functor is of degree $\leq 1$.
 \item The composition of functors of degree $\leq k$ and $\leq l$ is a functor of degree $\leq kl$.
 \item Let $T:\mathcal{A}\rightarrow\mathcal{B}$ and $R:\mathcal{C}\rightarrow\mathcal{B}$ be of degree $\leq k$ and $\leq l$, respectively. The level-wise sum $$T\oplus R :\mathcal{A}\times \mathcal{C}\rightarrow \mathcal{B}$$ is polynomial of degree  $\leq \operatorname{max}\{k,l \}$
\end{enumerate}
\end{prop}

\begin{example}Let $\operatorname{Mod}(\IZ)$ be the category of finitely generated $\IZ$-modules. An example of a degree $\leq k$ functor is the $k$-fold tensor product
\begin{displaymath}
\begin{aligned}
\otimes: &\operatorname{Mod}(\IZ)^k &&\rightarrow &&&\operatorname{Mod}(\IZ)\\
&(A_1,...,A_k) &&\mapsto &&&\bigotimes_i A_i.
\end{aligned}	
\end{displaymath}\end{example}Schur functors give examples of polynomial functors. Note that the following definitions also make sense for general commutative rings, but we are going to restrict our presentation to the category of graded rational vector spaces ${Vect_*}(\IQ).$ Schur functor are treated for example in \cite{algop}. We could not find any literature on Schur multifunctors and hence state the facts we need here. 

Let $\mathscr{M}=\mathscr{M}(n)\in {Vect_*}(\IQ)$, $n\geq 0$ be a sequence of $\IQ[\Sigma_n]$-modules. We refer to them as $\Sigma_n$-modules but implicitly use the $\IQ[\Sigma_n]$-modules structure, in particular $\otimes_{\Sigma_n}$ refers to the tensor product over $\IQ[\Sigma_n]$. The \emph{Schur functor} given by $\mathscr{M}$ is defined to be the endofunctor of ${Vect_*}(\IQ)$ induced by
$$\mathscr{M}(V)=\bigoplus_k \mathscr{M}(k)\otimes_{\Sigma_k}V^{\otimes k} \text{ for all }V\in {Vect_*}(\IQ),$$
where $V^{\otimes k}$ is the left $\Sigma_k$-module with action induced by the permutation of the factors by the inverse (with sign according to the Koszul sign convention). Note that $\mathscr{M}(0)$ is just a constant summand. A Schur functor $\mathscr{M}$ with $\mathscr{M}(l)$ trivial for $l>k$ is a polynomial functor of degree $\leq k$.

Let $\eta=(n_1,...,n_l),$ $n_1,...,n_l\geq 0$ be a multiindex. Throughout this article we will assume all multiindices to have non-negative entries. We use the following conventions:\begin{displaymath}
\begin{aligned}
	|\eta| &=\sum_{i=1}^l n_i,\\
	\ell(\eta) &=l,\\
	\mu+\eta &=(m_1+n_1,...,m_l+n_l) \text{, for $\mu=(m_1,...,m_l)$},\\
	(V_i)^{\otimes \eta} &=V_1^{\otimes n_1}\otimes ... \otimes V_l^{\otimes n_l} \text{, where }(V_i)\in ({Vect_*}(\IQ))^l,\\
	\Sigma_\eta &= \Sigma_{n_1}\times ... \times \Sigma_{n_l}.
\end{aligned}
\end{displaymath}
Consider a sequence of $\IQ[\Sigma_\eta]$-modules $\mathscr{N}=\mathscr{N}(\eta)\in {Vect_*}(\IQ)$, $\ell(\eta)=l.$ As before, we refer to them as $\Sigma_\eta$-modules. We define the \emph{Schur multifunctor} given by $\mathscr{N}$ on objects by
$$\mathscr{N}(V_1,...,V_l)=\bigoplus_{\ell(\eta)=l} \mathscr{N}(\eta)\otimes_{\Sigma_\eta}(V_i)^{\otimes \eta}\text{ for all }(V_i)\in ({Vect_*}(\IQ))^l.$$
 Similarly a Schur multifunctor $\mathscr{N}$ is a polynomial of degree $\leq k$ if $\mathscr{N}(\eta)$ is trivial for $|\eta|>k$.
\begin{example}
Consider Schur functors $\mathcal{N}_i:Vect_*({\IQ})\rightarrow Vect_*({\IQ}),$ $i=1,...,l$ The tensor product $$\bigotimes \mathcal{N}_i:Vect_*({\IQ})^l\rightarrow Vect_*(\IQ)$$ is a Schur multifunctor with $(\bigotimes \mathcal{N}_i)(\eta)=\bigotimes\mathcal{N}_i(n_i).$
\end{example}
We define the tensor product of $\mathscr{M}=\mathscr{M}(\mu)$ and $\mathscr{N}=\mathscr{N}(\eta)$ as $$\mathscr{M}\otimes \mathscr{N}(\nu):= \bigoplus_{\substack{\mu'+\eta'=\nu}}\operatorname{Ind}^{\Sigma_\nu}_{\Sigma_{\mu'}\times \Sigma_{\eta'}}\mathscr{M}(\mu')\otimes \mathscr{N}(\eta').$$
The Schur functor defined by this tensor product is indeed (up to natural isomorphism) the level-wise tensor product of the two functors, as we see by the isomorphisms for $\mu'+\eta'=\nu$ with $\ell(\mu')=\ell(\eta')=\ell(\nu)$
\begin{displaymath}
\begin{aligned}
&\left( \operatorname{Ind}^{\Sigma_\nu}_{\Sigma_{\mu'}\times \Sigma_{\eta'}}\mathscr{M}(\mu')\otimes \mathscr{N}(\eta') \right)\otimes_{\Sigma_\nu}(V_i)^{\otimes \nu} \\&= \left( \mathscr{M}(\mu')\otimes \mathscr{N}(\eta') \otimes_{\Sigma_{\mu'}\times \Sigma_{\eta'}} \IQ[\Sigma_\nu] \right)\otimes_{\Sigma_\nu}(V_i)^{\otimes \nu} \\
&\cong \left( \mathscr{M}(\mu')\otimes \mathscr{N}(\eta') \right)\otimes_{\Sigma_{\mu'}\times \Sigma_{\eta'}}\left( (V_i)^{\otimes \mu'}\otimes (V_i)^{\otimes \eta'}\right)\\
&\cong \left( \mathscr{M}(\mu')\otimes_{\Sigma_{\mu'}}(V_i)^{\otimes \mu'} \right) \otimes \left( \mathscr{N}(\eta') \otimes_{\Sigma_{\eta'}} (V_i)^{\otimes \eta'}\right).
\end{aligned}	
\end{displaymath}

The tensor powers of a $\mathscr{N}=\mathscr{N}(\eta)$ are (up to natural isomorphism) explicitly described by
\begin{equation}\label{schurtensorp}
\mathscr{N}^{\otimes r}(\nu)=\bigoplus \operatorname{Ind}_{\Sigma_{\eta_1}\times...\times \Sigma_{\eta_r}}^{\Sigma_\nu}\mathscr{N}(\eta_1)\otimes ... \otimes \mathscr{N}(\eta_r),
\end{equation}
where the sum runs over all $r$-tuples $(\eta_1,...,\eta_r),$ where $\ell(\eta_i)=l,$ such that $\Sigma_{i=1}^r \eta_i=\nu$.

Now consider a Schur functor $\mathscr{M}=\mathscr{M}(m)$ and a Schur multifunctor $\mathscr{N}=\mathscr{N}(\eta)$. The composition is a Schur multifunctor isomorphic to the Schur multifunctor given by
\begin{equation}\label{schurcomp}
(\mathscr{M}\circ \mathscr{N})(\nu)=\bigoplus_{r} \mathscr{M}(r) \otimes_{\Sigma_r} \bigoplus \operatorname{Ind}_{\Sigma_{\eta_1}\times...\times \Sigma_{\eta_r}}^{\Sigma_\nu}\mathscr{N}(\eta_1)\otimes ... \otimes \mathscr{N}(\eta_r),  
\end{equation}
where the second sum runs over all $r$-tuples $(\eta_1,...,\eta_r),$ where $\ell(\eta_i)=l,$ such that $\Sigma_{i=1}^r \eta_i=\nu$. The action of $\Sigma_r$ is by permuting the tuples $(\eta_1,...,\eta_r)$ by the inverse. Indeed as we check using (\ref{schurtensorp}):
\begin{displaymath}
\begin{aligned}
(&\mathscr{M}\circ \mathscr{N})(V_i) = \bigoplus_r \mathscr{M}(r) \otimes_{\Sigma_r} \mathscr{N}(V_i)^{\otimes r}\cong \bigoplus_r \mathscr{M}(r) \otimes_{\Sigma_r} \mathscr{N}^{\otimes r}(V_i)\\
&\cong \bigoplus_{r,\nu} \mathscr{M}(r) \otimes_{\Sigma_r} \left( \mathscr{N}^{\otimes r}(\nu)\otimes_{\Sigma_\nu} (V_i)^{\otimes \nu} \right)\\
&\cong \bigoplus_{r,\nu}  \mathscr{M}(r) \otimes_{\Sigma_r} \left( \bigoplus_{\Sigma_{i=1}^r \eta_i=\nu} \operatorname{Ind}_{\Sigma_{\eta_1}\times...\times \Sigma_{\eta_r}}^{\Sigma_\nu}\mathscr{N}(\eta_1)\otimes ... \otimes \mathscr{N}(\eta_r)\otimes_{\Sigma_\nu} (V_i)^{\otimes \nu} \right)
\\
&\cong \bigoplus_\nu  \bigoplus_r\mathscr{M}(r) \otimes_{\Sigma_r} \left( \bigoplus_{\Sigma_{i=1}^r \eta_i=\nu} \operatorname{Ind}_{\Sigma_{\eta_1}\times...\times \Sigma_{\eta_r}}^{\Sigma_\nu}\mathscr{N}(\eta_1)\otimes ... \otimes \mathscr{N}(\eta_r) \otimes_{\Sigma_\nu} (V_i)^{\otimes \nu}\right) .
\end{aligned}
\end{displaymath}

\begin{rem}
We will later use Schur (multi)functors with domain the category of rational vector spaces - just consider them as graded rational vector spaces concentrated in degree $0$.
\end{rem}

\section{Automorphisms of hyperbolic modules over the integers}\label{automorphisms of hyperbolic modules}

In this section we review hyperbolic modules in the sense of \cite{MR632404} in the special case with ground ring the integers. Fix a $\lambda\in \{+1,-1\}.$ Let $\Lambda\subset \IZ,$ be an additive subgroup, called the \emph{form parameter}, such that \begin{equation}\label{Lambda}
\{ z - \lambda z | z \in \IZ\}\subset \Lambda \subset \{ z\in \IZ|z=-\lambda z\}. 
\end{equation}
\begin{defi}[\cite{MR632404}]
A $\Lambda$-quadratic module is a pair $(M,\mu)$, where $M$ is a $\IZ$-module and $\mu$ is a bilinear form, i.e. a homomorphism $$\mu:M\otimes M\rightarrow \IZ.$$\end{defi}
To a $\Lambda$-quadratic module $(M,\mu)$ we associate a \emph{$\Lambda$-quadratic form} $$q_\mu:M\rightarrow \IZ/\Lambda\text{, } q_\mu(x)=[\mu(x,x)]$$ and a \emph{$\lambda$-symmetric bilinear form} $$\langle - , - \rangle_\mu:M\otimes M\rightarrow \IZ,$$ defined by $\langle x , y \rangle_\mu=\mu(x,y)+\lambda \mu(y,x).$ ($\lambda$-symmetric bilinear form like this are called even.) We call a finitely generated projective $\Lambda$-quadratic module $(M,\mu)$ \emph{non-degenerate}, if the map $$M\rightarrow M^*\text{, }x\mapsto \langle x,-\rangle_\mu $$ is an isomorphism.

Denote by $Q^\lambda(\IZ,\Lambda)$ the category non-degenerate $\Lambda$-quadratic modules and morphisms linear maps respecting the associated $\lambda$-symmetric bilinear form and the associated $\Lambda$-quadratic form.

Given a finitely generated $\IZ$-module $M$, we define a non-degenerate $\Lambda$-quadratic module $H(M)=(M\oplus M^*,\mu_M),$ where $\mu_M((x,f),(y,g))=f(y).$ We call $H(M)$ the \emph{hyperbolic module} on $M.$ A $\Lambda$-quadratic module is called \emph{hyperbolic}, if it is isomorphic to $H(N)$ for some finitely generated $\IZ$-module $N$. Let $\{e_i\}$ be the standard basis for $\IZ^g$ and $\{f_i\}$ the dual basis of $(\IZ^g)^*.$ Using this basis we consider the automorphisms of the $\Lambda$-quadratic module $H(\IZ^g)$ as a subgroup of $Gl_{2g}(\IZ).$ The subgroups can be described as follows:

\begin{prop}[{\cite[Cor. 3.2.]{MR632404}}]
The automorphisms of $H(\IZ^g)$ in $Q^{\lambda}(\IZ,\Lambda)$ are isomorphic to the subgroup of $Gl_{2g}(\IZ)$ consisting of matrices $\begin{pmatrix}
A & B   \\
C & D \end{pmatrix}$ such that
\begin{displaymath}
	\begin{aligned}
		&D^TA + \lambda B^TC=1 \\
		&D^TB+ \lambda B^T D=0 \\
		&A^t C + \lambda C^T A=0\\
		&\text{$C^T A$ and $D^TB$ have diagonal entries in $\Lambda.$}
	\end{aligned}
\end{displaymath}

\end{prop}

Note that if $\lambda=1$ we necessarily have $\Lambda=0.$ When $\lambda=-1,$ the condition (\ref{Lambda}) implies that $2\IZ\subset \Lambda \subset \IZ,$ thus we have the two cases $\Lambda=\IZ$ and $\Lambda=2\IZ.$ Thus we can list the automorphisms of hyperbolic modules:
\begin{enumerate}\item[]
	\item When $\lambda=1$ and $\Lambda=0,$ then $\Aut(H(\IZ^g))=O_{g,g}(\IZ)$ in $Q^{1}(\IZ,0).$ 
\item When $\lambda=-1$ and $\Lambda=\IZ,$ then $\Aut(H(\IZ^g))=Sp_{2g}(\IZ)$ in $Q^{-1}(\IZ,\IZ).$
\item When $\lambda=-1$ and $\Lambda=2\IZ,$ then $\Aut(H(\IZ^g))$ in $Q^{-1}(\IZ,2\IZ)$ is the subgroup of $Sp_{2g}(\IZ)$ described as:
$$\left \{\begin{pmatrix}
A & B   \\
C & D \end{pmatrix}\in Sp_{2g}(\IZ)| C^TA \text{ and } D^TB \text{ have even entries at the diagonal} \right \}.$$
\end{enumerate}

Let $N$ be a ($d-1$)-connected $2d$-manifold. Wall \cite{MR0145540} has shown that the automorphisms of the homology realized by diffeomorphisms are the automorphisms of a $\Lambda$-quadratic module with underlying $\IZ$-module $H_d(N).$ Later we show a similar statement for connected sums of products of spheres. For this we need a slight variation of $\Lambda$-quadratic modules.

Let $n\in\IN$ and $\Lambda\subset \IZ,$ be an additive subgroup, such that \begin{equation*}
\{ z - (-1)^{\frac{n}{2}}z | z \in \IZ\}\subset \Lambda \subset \{ z\in \IZ|z=-(-1)^{\frac{n}{2}} z\}. 
\end{equation*} Note that in the case that $n$ is odd $\Lambda$ has to be the trivial group, in fact in that case it will not be part of the following definitions at all, but we keep it in the notion for convenience.
\begin{defi}A \emph{graded $\Lambda$-quadratic module} is a pair $(M,\mu)$, where $M$ is a graded $\IZ$-modules and $\mu$ bilinear $n$-pairing, i.e. a degree $0$ homomorphism 
$$\mu: M\otimes M \rightarrow \IZ[n].$$
\end{defi}
We associate to $(M,\mu)$ a \emph{symmetric bilinear $n$-pairing}
$$\langle - , - \rangle_\mu:M\otimes M\rightarrow \IZ[n],$$ defined by $\langle x , y \rangle_\mu=\mu(x,y)+(-1)^{|x||y|} \mu(y,x).$
When $n=2d$, we associate a \emph{$\Lambda$-quadratic form} $$q_\mu:M_d\rightarrow \IZ/\Lambda\text{, } q_\mu(x)=[\mu(x,x)]$$  We call a finitely generated projective graded $\Lambda$-quadratic module $(M,\mu)$ non-degenerate, if the map $$M\rightarrow \hom (M,\IZ[n])\text{, }x\mapsto \langle x,-\rangle_\mu $$ is an isomorphism.

For $n$ even, we define $Q_*^{n}(\IZ,\Lambda)$ to be the category whose objects are non-degenerate graded $\Lambda$-quadratic modules $(M,\mu)$, where $(M,\mu)$ is finitely generated as a $\IZ$-module and the morphisms respect $q_\mu$ and $\langle-,-\rangle_\mu.$

For $n$ odd, we define $Q_*^{n}(\IZ,\Lambda)$ to be the category whose objects are non-degenerate graded $\Lambda$-quadratic modules $(M,\mu)$, where $(M,\mu)$ is finitely generated as a $\IZ$-module and the morphisms respect $\langle-,-\rangle_\mu.$

Let $Q_+^{n}(\IZ,\Lambda)$ be the full subcategory with objects concentrated in positive degrees and hence necessarily concentrated in degrees $1,...,n-1$.

For $0<p_i\leq q_i,$ $i\in I$ such that $p_i+q_i=n$ and $|I|$ finite, we define a graded $\Lambda$-quadratic module $$H_I=H(\IZ^{g_1}[1]\oplus ...\oplus \IZ^{g_{\lfloor n/2 \rfloor}}[\lfloor n/2 \rfloor]\oplus \hom_{\IZ}(\IZ^{g_{\lfloor n/2 \rfloor}}[\lfloor n/2 \rfloor ],\IZ[n])\oplus ...\oplus\hom_{\IZ}(\IZ^{g_{1}}[1],\IZ[n])),$$ where $$g_k = \{i\in I|p_i=k\}.$$ Denote by $\{a_i\}$ standard basis for $\IZ^{g_1}[1]\oplus ...\oplus \IZ^{g_{\lfloor n/2 \rfloor}}[\lfloor n/2\rfloor]$ and by $\{b_i\}$ the dual basis. The pairing $\mu_{H_I}=\mu_I$ then given by $$\mu_I(a_i,b_j)=b_j(a_i)=\delta_{i,j}\text{ and } \mu_I(a_i,a_j)=\mu_I(b_i,b_j)=0.$$

Denote by $\Gamma_I=\Aut(H_I)$ in $Q_+^{n}(\IZ,\Lambda).$ We get the following cases:
\begin{enumerate}\item[]
	\item When $n$ is odd we get $$\Gamma_I\cong\prod_{k=1}^{\lfloor n/2\rfloor}Gl_{g_k}(\IZ).$$
	\item When $n=2d$ and $d$ is even we necessarily have $\Lambda=0$ and $$\Gamma_I\cong O_{g_d,g_d}(\IZ)\times \prod_{k=1}^{n/2-1}Gl_{g_k}(\IZ).$$
	\item Similarly for $n=2d$ with $d$ odd the only cases are $\Lambda=\IZ,2\IZ$ and we just get a products of general linear groups and $Sp_{2g_d}(\IZ)$ respectively the subgroup described in the list of automorphism groups above under point 3.
\end{enumerate}

Berglund and Madsen call group $G$ \emph{rationally} perfect, if $H^1(G;V)=0$ for any finite dimensional rational $G$-representation $V.$ We later need that the automorphism groups of graded hyperbolic modules are rationally perfect.

\begin{lem}\label{gammaIh1}The groups $\Gamma_I$ are rationally perfect.
\end{lem}

\begin{proof}We begin by observing that being rationally perfect is stable under group extensions, i.e. if in a group extension $$0\rightarrow K \rightarrow G \rightarrow C \rightarrow 0,$$ $K$ and $C$ are rationally perfect, then so is $G.$ This follows from the Lyndon spectral sequence, since $H^1(C;H^0(K;V))$ and $H^0(C;H^1(K;V))$ are trivial for any finite dimensional rational $G$-representation $V.$ In particular products of rationally perfect groups are rationally perfect. Moreover we observe that finite groups are rationally perfect.

It follows from Borel's work on the cohomology of arithmetic groups that the automorphism groups $\Aut(H(\IZ^g))$ of the hyperbolic modules $H(\IZ^g)$ in $Q^\lambda(\IZ,\Lambda)$ are rationally perfect for $g\geq 2$ (see e.g. \cite{berg13}).

In \cite{bassmilnorserre} it is shown that $Sl_g(\IZ)$ is rationally perfect for $g\geq 3$. Since $Sl_g(\IZ)$ is an index two normal subgroup of $Gl_g(\IZ),$ this now also implies that $Gl_g(\IZ)$ is rationally perfect for $g\geq 3.$

Recall that the $\Gamma_I$ are products of $\Aut(H(\IZ^k))$ and $Gl_l(\IZ)$-s. Since $Gl_1(\IZ)$ is finite and hence rationally perfect, to finish the proof we have to show that $\Aut(H(\IZ))$ and $Gl_2(\IZ)$ are rationally perfect.

The group $Sl_2(\IZ)$ is an extension $$0\rightarrow C_2\rightarrow Sl_2(\IZ) \rightarrow C_2 * C_3 \rightarrow 0$$ and $C_2 * C_3$ can be seen to be rationally perfect using a Mayer-Vietoris argument. Hence $Sl_2(\IZ)$ and also $Gl_2(\IZ)$ are rationally perfect.

The group $\Aut(H(\IZ))$ in $Q^{-1}(\IZ,\IZ)$ is $Sp_2(\IZ),$ which is isomorphic to $Sl_2(\IZ).$ For $\Lambda=2\IZ$ it follows that $\Aut(H(\IZ))$ is rationally perfect because it is a finite index subgroup of $Sp_2(\IZ).$

Recall that for $\lambda=1,$ we necessarily have $\Lambda=0$ and $\Aut(H(\IZ))\cong O_{1,1}(\IZ)\cong C_2\times C_2$ and hence we are done.
\end{proof}

\section{Van der Kallen's and Charney's homological stability results}\label{van der kallens and charneys homological stab results}
In this section we recall van der Kallen's homological stability for general linear groups and Charney's homological stability for automorphisms of hyperbolic quadratic modules. We combine them to homological stability for the $\Gamma_I$ defined above with certain coefficient systems induced by polynomial functors.
\begin{rem}
Charney's results hold for Dedekind domains with involutions and van der Kallen's for associative rings with finite stable range, but we restrict our presentation to $\IZ$ (with trivial involution).	
\end{rem}

We begin by reviewing the notion of coefficient systems as discussed in \cite{dwyer}.
A coefficient system for $\{ Gl_g(\IZ) \}_{g \geq 1}$ is a sequence of $Gl_g(\IZ )$-modules $\{ \rho_g\}_{g \geq 1}$ together with $ Gl_g(\IZ) $-maps $F_g:\rho_g\rightarrow I^* (\rho_{g+1})$, where $I^*$ denotes the restriction via the upper inclusion $I:Gl_g(\IZ)\hookrightarrow Gl_{g+1}(\IZ)$. We denote the system by $\rho$ and call the maps $F_g$ structure maps. A map of coefficient systems $\rho$ and $\rho'$ is a collection of $Gl_g(\IZ)$-maps $\{\tau_g\}_{g \geq 1}$ such that they commute with the structure maps.  The level-wise kernels and  cokernerls are again coefficient systems with the obvious structure maps.
Denote by $J:Gl_g(\IZ)\rightarrow Gl_{g+1}(\IZ)$ the lower inclusion map. For a coefficient system $\rho$ we define the shifted system $\Sigma \rho$ by $\Sigma \rho_g := J^* (\rho_{g+1})$ with structure maps $\Sigma F_g:= J^* F_{g+1}: J^* (\rho_{g+1})\rightarrow I^*J^* (\rho_{g+2})$ . Denote by $s_g\in Gl_g(\IZ)$ the element permuting the last two standard basis elements.  We call a coefficient system central, if $s_{g+2}$ acts trivially on the image of $F_{g+1}F_g:\rho_g\rightarrow \rho_{g+2}$. Denote by $e_{g-1,g}\in Gl_g(\IZ)$ the element sending all but the $g$-th standard basis element to itself and the $g$-th $e_g$ to $e_{g-1}+e_g$. We call a central coefficient system strongly central, if $e_{g+1,g+2}$ acts trivially on the image of $F_{g+1}F_g:\rho_g\rightarrow \rho_{g+2}.$

Let $c_g\in Gl_g(\IZ)$ $(g>1)$ be the element sending the i-th standard basis element to the $(i+1)$-st and the $g$-th to the first. 

Denote by $\mu(c_g)$ the multiplication from the left by $c_g$. Then the following holds:
\begin{lem}[{\cite[Lemma 2.1.]{dwyer}}]
Let $\rho$ be a central coefficient system. Then we have a map of coefficient systems $\tau:\rho\rightarrow\Sigma\rho$, defined by
$$\tau_g:\rho_g\xrightarrow{F_g} I^*(\rho_{g+1})\xrightarrow{\mu (c_{g+2})} {J^*}(\rho_{g+1})=\Sigma \rho_g.$$
\end{lem}
We say that a central coefficient system $\rho$ splits, if $\Sigma \rho$ is isomorphic to $\rho\oplus \text{coker}(\tau)$ via $\tau$. We then denote $\text{coker}(\tau)$ by $\Delta \rho$. We now define the notion of degree of a strongly central coefficient system $\rho$ inductively. We say it has degree $k<0$, if it is constant and for $k\geq 0$ we say that it has degree $\leq k$, if $\Sigma \rho$ splits and $\Delta \rho$ is a strongly central coefficient system of degree $k-1$.

\begin{thm}[{Van der Kallen \cite[p. 291]{vdk}}]\label{vdkstab}
 Let $\rho$ be a strongly central coefficient system of degree $\leq k$, then $$H_i(Gl_g(\IZ),\rho_g))\rightarrow H_i(Gl_{g+1}(\IZ),\rho_{g+1})$$ is an isomorphism for $g>2i+k+2$ and en epimorphism for $g\geq2i+k+2$.
\end{thm}

Denote by $\lambda_g$ the standard representation of $Gl_g(\IZ)$ on $\IZ^g$ and by $\bar \lambda_g$ the action by the inverse transposed on $\IZ^g.$ Let $\mathcal{A}$ be an abelian category. Given a functor $$T:\operatorname{Mod}(\IZ)\times\operatorname{Mod}(\IZ)\rightarrow \mathcal{A},$$ we define a coefficient system $\{T(\lambda_g,\bar \lambda_g)\}_{g\geq 1}$ with structure maps induced by the standard inclusions and actions induced by $\lambda_g$ and $\bar \lambda_g$.

\begin{lem}[{Compare \cite[5.5.]{vdk} and \cite[Lemma 3.1.]{dwyer}}]
If $$T:\operatorname{Mod}(\IZ)\times\operatorname{Mod}(\IZ)\rightarrow \mathcal{A}$$ is a polynomial functor of degree $\leq k$, then $\{ T(\lambda_g,\bar \lambda_{g})\}_{g\geq 1}$ is a strongly central coefficient system of degree $\leq k$. 
\end{lem} 

Denote now $G_g$ the automorphisms of $H(\IZ^g)$ in $Q^{\lambda}(\IZ,\Lambda).$ And denote by $e_1,...,e_g$ the standard basis for $\IZ^g$ and by $f_1,...,f_g$ the dual basis of $(\IZ^g)^*.$ We see $G_g$ as a subgroup of $Gl_{2g}(\IZ)$, by considering the elements of $G_g$ as $2g\times 2g$-matrices acting on $H(\IZ^g)\cong \IZ^{2g}.$ We define the upper inclusion $$I:G_g\rightarrow G_{g+1}\text{, }\begin{pmatrix}
A & B   \\
C & D \end{pmatrix}\mapsto \begin{pmatrix}
A & 0 & B & 0\\
0 & 1 & 0 & 1\\
C & 0 & D & 0\\
0 & 1 & 0 & 1\\
\end{pmatrix} $$ and similarly the lower inclusion $J:G_g\rightarrow G_{g+1}.$ The definition of a coefficient system is very similar to the one for $GL_g(\IZ)$ and we only briefly summarize it.
A coefficient system for $\{ G_g \}_{g \geq 1}$ is a sequence of $G_g$-modules $\{ \rho_g\}_{g \geq 1}$ together with $ G_g $-maps $F_g:\rho_g\rightarrow I^* (\rho_{g+1}).$ We denote a coefficient system again by $\rho$ and let maps of coefficient systems be as above. The shifted coefficient system $\Sigma \rho$ is the restriction via the lower inclusion as above. A coefficient system is called central if $c_{g+2}\circ c_{2g+4}$ acts trivially on the image of $F_{g+1}F_g:\rho_g\rightarrow \rho_{g+2}$. For a central coefficient system we define the map of coefficient systems $\tau:\rho\rightarrow\Sigma\rho$, by
$$\tau_g:\rho_g\xrightarrow{F_g} I^*(\rho_{g+1})\xrightarrow{\mu (c_{g+2}\circ c_{2g+4})} {J^*}(\rho_{g+1})=\Sigma \rho_g.$$ We call a central coefficient system $\rho$ split, if $\tau$ is injective and $\Sigma \rho \cong \tau(\rho)\oplus \operatorname{coker}(\tau).$ For a central coefficient system we define the degree inductively: We say it has degree $k<0$, if it is constant and for $k\geq 0$ we say that it has degree $\leq k$, if $\Sigma \rho$ splits and $\Delta  \rho=\operatorname{coker}(\tau)$ is a strongly central coefficient system of degree $\leq k-1$.

\begin{thm}[{Charney \cite[Theorem 4.3.]{MR885099}}]\label{charneystab}
 Let $\rho$ be a central coefficient system of degree $\leq k$, then $$H_i(G_g,\rho_g))\rightarrow H_i(G_{g+1},\rho_{g+1})$$ is an isomorphism for $g>2i+k+4$ and an epimorphism for $g\geq 2i+k+4$.
\end{thm}

Again we get a central coefficient system of degree $\leq k$ by considering the standard $G_g$-action $\lambda_{g,g}$ on $H(\IZ^g)\cong \IZ^{2g}$, induced by the inclusion $G_g\subset Gl_{2g}(\IZ).$ Let $\mathcal{A}$ be an abelian category. Given a polynomial functor $$T:\operatorname{Mod}(\IZ)\rightarrow \mathcal{A},$$ of degree $\leq k$ then $\{T(\lambda_{g,g})\}_{g\geq 1}$ is a central coefficient system of degree $\leq k$ for $\{G_g \}_{g\geq 1}$.

Now we combine the homological stability results above to a result for the groups $\Gamma_I$ defined in the previous section.

Denote by $\lambda_I$ the \emph{standard representation of $\Gamma_I$} on $(\IZ^{g_1},...,\IZ^{2 g_{n/2}},...,\IZ^{g_{n-1}})$ induced by the inclusion $\Gamma_I\subset \prod_{i=1}^{\lfloor n/2\rfloor} Gl_{r_i}(\IZ),$ where we consider $\IZ^{2 g_{n/2}}$ to be the empty set if $n$ is odd and 
$$r_k = \begin{cases} 2g_k=2|\{i\in I|p_i=k\}| &\text{ if }k=n/2\\ g_k=|\{i\in I|p_i=k\}| &\text{ if }k <n/2.\end{cases}$$ More explicitly, we assume that an automorphisms $$A=(A_1,...,A_{n/2})\in \Gamma_I\subset \prod_{i=1}^{\lfloor n/2\rfloor} Gl_{r_i}(\IZ),$$ acts by matrix multiplication of $A_i$ on $\IZ^{g_i}$ for $i\leq n/2$ and by multiplication by the inverse transposed of $A_i$ on $\IZ^{g_{n-i}}$ for $i>n/2.$

Given a functor to an abelian category $T:Mod(\IZ)^{n-1}\rightarrow \mathcal{A}$, we get a $\Gamma_I$-module $T(\IZ^{g_1},...,\IZ^{2g_{n/2}},...,\IZ^{g_{n-1}})$ with the induced action. We denote this $\Gamma_I$-module by $T(\lambda_I).$ For a fixed $p\in\IN$, such that $0<p\leq n/2,$ denote by $\Gamma_{I'}$ the automorphisms of $H_{I'},$ where $I'=I\cup \{i'\}$ with $p_{i'}=p$ and $q_{i'}=n-p$. We define the \emph{stabilization map} $$\sigma_{p,n-p}:H_i(\Gamma_I,T(\lambda_I))\rightarrow H_i(\Gamma_{I'},T(\lambda_{I'}))$$ to be the map induced by the obvious upper inclusion $I_{p,q}:\Gamma_I\rightarrow\Gamma_{I'}$ and $T(I_{p,q}).$

\begin{prop}\label{general algebraic homological stability result}Let $\mathcal{A}$ be some abelian category, $T:Mod(\IZ)^{n-1}\rightarrow \mathcal{A}$ be a polynomial functor of degree $\leq k$. The stabilization map $$\sigma_{p,n-p}:H_i(\Gamma_I,T(\lambda_I))\rightarrow H_i(\Gamma_{I'},T(\lambda_{I'}))$$ induces an isomorphism for $g_p>2i+k+2$ when $2p\neq n$ and $g_p>2i+k+4$ if $2p=n$ and an epimorphism for $g_p\geq 2i+k+2$ respectively $g_p\geq2i+k+4.$
\end{prop}

\begin{proof} Denote by $\Gamma_{g_p}$ the summand of $\Gamma_I\subset \prod_{i=1}^{\lfloor n/2\rfloor} Gl_{r_i}(\IZ)$ that sits in $Gl_{r_p}(\IZ)$. Let $\Gamma=\Aut(H_{\bar I}),$ where $\bar I= I\smallsetminus \{i\in I| p_i=p\}.$ Note that $\Gamma\oplus \Gamma_{g_p}= \Gamma_I$ and $\Gamma_{I'}=\Gamma_{g_p+1}\oplus \Gamma,$ where $\Gamma_{g_p+1}$ is defined analogous to $\Gamma_{g_p}.$
Consider the functor $$\mathfrak{I}_{p,n-p}:\begin{cases}
\operatorname{Mod}(\IZ)\rightarrow Mod(\IZ)^{n-1}\text{ if $2p=n$}\\
\operatorname{Mod}(\IZ)\times\operatorname{Mod}(\IZ) \rightarrow Mod(\IZ)^{n-1}	\text{ otherwise,}
\end{cases}$$ defined by sending a module $M$ to $(\IZ^{g_1},...,M,...,\IZ^{g_{n-1}})$, where the $M$ sits at the $(n/2)$-th summand and a pair $(M,N)$ to $(\IZ^{g_1},...,M,...,N,...,\IZ^{g_{n-1}})$, where the $M$ sits at the $p$-th summand the $M$ sits at the $(n-p)$-th summand. This functor is clearly additive and hence of degree $\leq 1.$ This implies that the composition $T\circ \mathfrak{I}_{p,n-p}$ is of degree $\leq k$ and we get a (strongly) central coefficient system of degree $\leq k$ for $$\Gamma_{g_p}=\begin{cases} G_{g_p} & \text{ if $2 p=n$}  \\
	Gl_{g_p}(\IZ) & \text{otherwise.} 
\end{cases}$$
This implies that the stabilization maps
$$H_i(\Gamma_{g_p},T\circ \mathfrak{I}_{n/2,n/2}(\lambda_{g_p,g_p}))\rightarrow H_i(\Gamma_{g_{p}+1},T\circ \mathfrak{I}_{n/2,n/2}(\lambda_{g_p+1,g_p+1}))$$ respectively
$$H_i(\Gamma_{g_p},T\circ \mathfrak{I}_{p,n-p}(\lambda_{g_p},\bar \lambda_{g_p}))\rightarrow H_i(\Gamma_{g_{p}+1},T\circ \mathfrak{I}_{p,n-p}(\lambda_{g_p+1},\bar \lambda_{g_p+1}))$$ are isomorphism and epimorphism in the ranges in the statement of the Proposition. Observing that the $T\circ \mathfrak{I}_{n/2,n/2}(\lambda_{g_p,g_p})$ respectively $T\circ \mathfrak{I}_{p,n-p}(\lambda_{g_p},\bar \lambda_{g_p})$ are precisely the restrictions of the $\Gamma_I$-representation to the subgroup $\Gamma_{g_p}$ and using Lyndon spectral sequence $$H_k(\Gamma_{\bar I},H_l(\Gamma_{g_p},T(\lambda_I))) \Rightarrow H_i(\Gamma_I,T(\lambda_I)),$$ the results follows by comparing spectral sequences.
\end{proof}

\section{On Mapping Class Groups}\label{on mapping class groups}
Denote by $$N=N_I=\left(\#_{i\in I} (S^{p_i}\times S^{q_i})\right)\smallsetminus \interior (D^n),$$ where $|I|<\infty,$ $3\leq p_i\leq q_i<2p_i-1$ and $p_i+q_i=n$ for $i\in I$.

Denote by $\Aut(\tilde{H}_*(N_I))$ the automorphisms of the graded group $\tilde{H}_*(N_I).$
In this section we study the map $$H_*:\pi_0 \; \aut_\partial (N_I)\rightarrow \Aut(\tilde{H}_*(N_I)).$$ In particular we are going to determine its image and show that the kernel is finite.

Denote by $in:\partial N\hookrightarrow N$ the inclusion of the boundary. We observe that $V_I=\bigvee_{i\in I} (S^{p_i}\vee S^{q_i})\subset N$ is a deformation retract and denote by $$\alpha_j:S^{p_j} \hookrightarrow \bigvee_{i\in I} (S^{p_i}\vee S^{q_i}) \text{ and } \beta_j:S^{q_j} \hookrightarrow \bigvee_{i\in I} (S^{p_i}\vee S^{q_i})$$ the canonical inclusions. We consider $in$ as an element of $\pi_{n-1}(\bigvee_{i\in I} (S^{p_i}\vee S^{q_i}))$ and we observe that it is given by $\sum_{i\in I} [\alpha_i,\beta_i]$. Denote by $$\begin{aligned}\langle-,-\rangle:&H_*(N)\otimes H_{n-*}(N) \rightarrow \IZ\\ &x\otimes y \mapsto (PD^{-1}(x)\cup PD^{-1}(y))([N,\partial N])\end{aligned}$$ the intersection form, where $PD^{-1}:H_*(N)\rightarrow H^{n-*}(N,\partial N)$ denotes the Poincar\'e duality isomorphisms and we evaluate on the fundamental class $[N,\partial N].$ The $\{ \alpha_i \}$ and $\{ \beta_i \}$ define a basis for $\tilde H_*(N)$ via the Hurewicz homomorphism, which we denote by $\{ a_i\}$ respectively $\{b_i\}$. Note that $b_i$ is dual to $a_i.$

In the case $n=2d$ is even we need to recall a further piece of structure from Wall's classification of highly connected even-dimensional manifolds \cite{MR0145540}. The elements $x\in H_{d}(N)$ can be represented by embedded $S^{d}.$ Denote by $\nu_x\in \pi_{d-1}(SO(d))$ the clutching function of the normal bundle of this embedding. It is independent of the choice of embedding, since homotopic embeddings are isotopic in this case. This defines a function $$q:H_{d}(N)\rightarrow \pi_{d-1}(SO(d))\text{, }x\mapsto [\nu_x] .$$ Denote by $\iota_{d}$ the class of the identity in $\pi_{d}(S^{d})$ and by $$\partial:\pi_d(S^d)\rightarrow \pi_{d-1}(SO(d))$$ the boundary map in the fibration $SO(d)\rightarrow SO(d+1)\rightarrow S^d.$ The function $q$ satisfies:
$$\langle x,x \rangle=HJq(x) \text{ and } q(x+y)=q(x)+q(y)+\langle x,y \rangle\partial \iota_d,$$ where $\pi_{d-1}(SO(d))\xrightarrow{J} \pi_{2d-1}(S^d)\xrightarrow{H}\IZ$ denote the $J$-homomorphism and the Hopf invariant. There is also a purely homotopy theoretic description of $Jq$ in \cite[Section 8]{MR0148075}. 

Note that for $a_i,b_j\in H_{d}(N),$ we have $q(a_i)=q(b_j)=0.$ Hence $\operatorname{Image}(q)$ is contained in the subgroup $\langle\partial\iota_d\rangle$ generated by $\partial\iota_d.$ The $J$-homomorphism restricts to an isomorphism $$J|\langle\partial\iota_d\rangle:\langle\partial\iota_d\rangle\rightarrow J(\langle\partial\iota_d\rangle)\cong \begin{cases} \langle [\iota_d,\iota_d]\rangle \xrightarrow[H]{\cong} 2\IZ &\text{if $d$ even}\\
0 &\text{if $d=1,3,7$ }\\
\langle [\iota_d,\iota_d]\rangle\cong \IZ/2\IZ &\text{if $d$ is odd and not $1,3$ or $7$,}
\end{cases}$$ where the second isomorphism is induced by the Hopf invariant.
Let $$\Aut(\tilde H_*(N),\langle-,-\rangle, Jq)\text{ and }\Aut(\tilde H_*(N),\langle-,-\rangle, q)$$ be the automorphisms of the reduced homology respecting the intersection form and the function $Jq$ ($q$ respectively). Note that $$\langle x, y \rangle=\mu(x,y)+(-1)^{|x||y|}\mu(y,x),$$ where $\mu(-,-)$ is determined by $$\mu(a_i,b_j)=\delta_{i,j} \text{ and } \mu(b_i,a_j)=\mu(a_i,a_j)=\mu(b_i,b_j)= 0.$$ Now let $$\Lambda=\begin{cases}0 & \text{ if $n=2d$ and $d$ is even}\\
\IZ & \text{ if $n=2d$ and $d$ is $3$ or $7$}\\
2\IZ &\text{ if $n=2d$ and $d$ is odd and not $3$ or $7$.} 

\end{cases}$$ 

Moreover $Jq=q_\mu,$ where $q_\mu$ is the $\Lambda$-quadratic form associated to $\mu,$ where we identify $\langle [\iota_d,\iota_d]\rangle$ with $\IZ$ and $\IZ/2\IZ$ respectively. It suffices to check this for the elements $a_i+b_i$ and for these $Jq(a_i+b_i)=[\iota_d,\iota_d]$ and $q_\mu(a_i+b_i)=1.$ By the discussion above we see that 
$$\Aut(\tilde H_*(N),\langle-,-\rangle, q)\cong \Aut(\tilde H_*(N),\langle-,-\rangle, Jq)\cong\Gamma_I= \Aut(H_I) \text{ in } Q_*^n(\IZ,\Lambda).$$

For a representative $f$ of $[f]\in \pi_0 (\aut_\partial (N))$ it is clear that $\tilde H_*(f)\in \Gamma_I$. We are now going to show that all elements of $\Gamma_I$ can be realized by a homotopy self-equivalence, fixing the boundary pointwise.

\begin{prop}\label{hmcg}
The group homomorphism
\begin{equation*}
\pi_0 (\aut_\partial (N))\xrightarrow{\tilde H_*} \Aut (\tilde H_* (N),\langle-,-\rangle, Jq )
\end{equation*}
is surjective and has finite kernel.
\end{prop}
\begin{proof}Compare \cite[Proof of Theorem 2.10]{berg12}. The cofibration $in:\partial N\hookrightarrow N$ induces a fibration 
$$\map_\partial (N,N)\rightarrow\map_*(N,N) \rightarrow \map_*(\partial N, N),$$ where $\map_\partial (N,N)$ is the fiber over $in$. Let $\map_\partial (N,N)$ and $\map_*(N,N)$ be based at the identity. Restricting the total space to invertible elements, we also get the following fibration: $$\aut_\partial (N)\rightarrow\aut_*(N) \rightarrow \map_*(\partial N, N).$$

We are going to analyze the long exact homotopy sequences $$\xymatrix{
...\ar[r]& \pi_1 (map_*(\partial N, N),in)\ar[r]\ar@{=}[d]& \pi_0( \aut_\partial(N) )\ar[r]\ar@{^{(}->}[d]& \pi_0 (\aut_*(N))\ar[r]\ar@{^{(}->}[d]& [\partial N,N]_*\ar@{=}[d]\\
...\ar[r]& \pi_1 (map_*(\partial N, N),in)\ar[r]& [N,N]_\partial \ar[r]& [N,N]_*\ar[r]& [\partial N,N]_* .}$$ We consider the monoid homomorphism $$\tilde H_*:[N,N]_*\rightarrow \operatorname{End}(\tilde H_*(N))$$ and show that it is onto and with finite kernel. Using the relative Hurewicz isomorphism, it is easy to see that $V_I\hookrightarrow \prod_{i\in I} (S^{p_i}\times S^{q_i})$ is $(2\operatorname{min}_{i\in I}\{ p_i \}-1)$-connected and hence more than $\operatorname{max}_{i\in I}\{ q_i \}$-connected. Thus we get an isomorphism of sets \begin{align*}&[N,N]_* \cong [V_I,V_I]_* \cong [V_I,\prod_{i\in I} (S^{p_i}\times S^{q_i})]_* \\
&\cong \prod_{(i,j)\in I\times I} [S^{p_i},S^{p_j}]_* \times \prod_{(i,j)\in I\times I} [S^{q_i},S^{q_j}]_* \times \prod_{(i,j)\in I\times I} [S^{q_i},S^{p_j}]_* \times \prod_{(i,j)\in I\times I} [S^{p_i},S^{q_j}]_* .
\end{align*} We write $I=\bigcup_l I_l,$ where $I_l=\{i|p_i=l\}.$
The only non-finite factors of the product above are \begin{equation}\label{nonfinitefactors}
	\prod_{l}\prod_{(i,j)\in I_l\times I_l} \left([S^{p_i},S^{p_j}]_*\times[S^{q_i},S^{q_j}]_*\right).
\end{equation} We identify $\operatorname{End}(\tilde H_*(N))\cong \prod \operatorname{Mat}_{r_l}(\IZ) ,$ where $r_l=\operatorname{rank}(H_l(N)),$ using the basis $\{a_i\}\cup\{b_i\}$. Note that for $l=n/2$ a $b_i$ becomes a $(r_l/2+i)$-th basis element. Denote by $\alpha^l_1,...,\alpha^l_{r_l}$ and $\beta^l_1,...,\beta^l_{r_l}$ the inclusions $S^l\hookrightarrow V_I$ and $S^{n-l}\hookrightarrow V_I$ respectively. There is a multiplicative section of $\tilde H_*$ \begin{equation}\label{section} \prod \operatorname{Mat}_{r_i}(\IZ)\rightarrow [N,N]_*,\;
	(M^l)=(m^l_{i,j})\mapsto f_{(M^l)}= \bigvee_{l=1}^{\lfloor n/2\rfloor} f_{M^l},
\end{equation} where $f_{M^l}:\bigvee_{i\in I_l} S^{p_i}\vee S^{q_i}\rightarrow V_I\text{, }$ is given by $$f_{M^l}= \begin{cases}
\bigvee_{i=1}^{r_l} (\sum^{r_l}_{j=1} m^l_{i,j}\alpha^l_j \vee \sum^{r_l}_{j=1} m^{n-l}_{i,j}\beta^l_j ) &\text{ if } l\neq n/2\\ \\
\bigvee_{i=1}^{r_l/2} (\sum^{r_l/2}_{j=1} m^l_{i,j}\alpha^l_j+\sum^{r_l}_{j=r_l/2+1} m^l_{i,j}\beta^l_{(j-r_l/2)})\\ \vee \bigvee_{i=r_l/2+1}^{r_l} (\sum^{r_l/2}_{j=1} m^l_{i,j}\alpha^l_j+\sum^{r_l}_{j=r_l/2+1} m^l_{i,j}\beta^l_{(j-r_l/2)}) &\text{ if } l=n/2.
\end{cases}$$ We observe that the image of this section is precisely the sub-monoid of $[N,N]_*$ corresponding to the non-finite factors (\ref{nonfinitefactors}). Hence we get that $\tilde H_*:[N,N]_*\rightarrow \Aut(\tilde H_*(N))$ is surjective and has finite kernel.
Restricting to the submonoids of invertible elements this implies upon using the section (\ref{section}) that $\pi_0(\aut_*(N))\rightarrow \operatorname{Aut}(\tilde H_*(N))$ is surjective with finite kernel. The image of $\pi_0(\aut_\partial(N))\rightarrow \pi_0(\aut_*(N))$ consists of the elements $[f]\in \pi_0(\aut_*(N)),$ such that $f\circ in\simeq in$ (we assume all homotopy equivalences in this proof to be pointed). Since (\ref{section}) restricts to a section $\operatorname{Aut}(\tilde H_*(N))\rightarrow \pi_0(\aut_*(N))$ we get that the image of $$\tilde H_*:\pi_0(\aut_\partial(N))\rightarrow \operatorname{Aut}(\tilde H_*(N))$$ is given by the $(M^l)$ such that $f_{(M^l)}\circ in\simeq in.$ Using the Hilton-Milnor Theorem we identify \begin{equation}\label{hiltonmilnoridentification}
 	\pi_{n-1}(N)\cong \pi \oplus \bigoplus_{l} \pi_{n-1}(\bigvee_{i\in I_l}(S^{p_i}\vee S^{q_l})),
\end{equation}where $\pi$ is some subgroup of $\pi_{n-1}(N).$ We observe that $$ f_{(M^l)}\circ in \simeq \sum_{i\in I} [f_{(M^l)}\circ\alpha_i,f_{(M^{l})}\circ\beta_i]\simeq \sum_{l=1}^{\lfloor n/2\rfloor}\sum_{i\in I_l} [f_{M^l}\circ\alpha^l_i,f_{M^l}\circ\beta^l_i],$$ i.e. that the action of $f_{(M^l)}$ respects the summands of the identification (\ref{hiltonmilnoridentification}). Thus it suffices to check that $f_{M^l}\circ \sum_{i\in I_l} [\alpha_i^l,\beta_i^l]\simeq \sum_{i\in I_l}[\alpha_i^l,\beta_i^l]$ for all $l.$ We use that left homotopy composition is distributive for suspensions \cite[p. 126]{MR0041435}, i.e. $(x+y)\circ \Sigma z\simeq x\circ \Sigma z+ y\circ \Sigma z $. For $l\neq n/2$ we calculate \begin{align*} f_{M^l}\circ \sum_{i=1}^{r_l} [\alpha_i^l,\beta_i^l] \simeq & \sum_{i=1}^{r_l} [f_{M^l}\circ\alpha^l_i,f_{M^l}\circ\beta^l_i] \simeq   \sum_{i,j,k} [m^l_{i,j}\alpha^l_j,m^{n-l}_{i,k}\beta^l_k] \\
  \simeq &    \sum_{i,j,k} m^l_{i,j}m^{n-l}_{i,k} [\alpha^l_j,\beta^l_k] \simeq    \sum_{j,k} ((M^l)^TM^{n-l})_{j,k} [\alpha^l_j,\beta^l_k].
 \end{align*} This expression is homotopic to $\sum_{i=1}^{r_l} [\alpha_i^l,\beta_i^l] $ if \begin{equation}\label{mcgconsition1}
(M^l)^TM^{n-l}=\id_{Mat_{r_l}{(\IZ)}}.	
\end{equation} For $l= n/2$ we write $$M^l=\begin{pmatrix}
A^l & B^l   \\
C^l & D^l 
\end{pmatrix}.$$ We calculate:\begin{align*} &f_{M^l}  \circ \sum_{i=1}^{r_l/2} [\alpha_i^l,\beta_i^l] \simeq  \sum_{i=1}^{r_l/2} [f_{M^l}\circ\alpha^l_i,f_{M^l}\circ\beta^l_i] \\ 
 \simeq & \sum_{i=1}^{r_l/2}\left[\sum^{r_l/2}_{j=1} (a^l_{i,j}\alpha^l_j + b^l_{i,j}\beta^l_{j}) ,\sum^{r_l/2}_{k=1} (c^l_{i,k}\alpha^l_k+ d^l_{i,k}\beta^l_{k})\right]\\
 \simeq& \sum_{i=1}^{r_l/2}\sum^{r_l/2}_{j=1}\sum^{r_l/2}_{k=1} a^l_{i,j} d^l_{i,k} [\alpha^l_j,\beta^l_{k}]            +\sum_{i=1}^{r_l/2}\sum^{r_l/2}_{j=1}\sum^{r_l/2}_{k=1} b^l_{i,j}c^l_{i,k} [\beta^l_{j},\alpha^l_k] \\
 +& \sum_{i=1}^{r_l/2}\sum^{r_l/2}_{j=1}\sum^{r_l/2}_{k=1}  a^l_{i,j}c^l_{i,k} [\alpha^l_j ,\alpha^l_k]            +\sum_{i=1}^{r_l/2}\sum^{r_l/2}_{j=1}\sum^{r_l/2}_{k=1}  b^l_{i,j}d^l_{i,k}[\beta^l_{j},\beta^l_{k}]\\
 \simeq& \sum^{r_l/2}_{j=1}\sum^{r_l/2}_{k=1} ( (A^l)^TD^l +(-1)^{n/2}((C^l)^TB^l ))_{j,k} [\alpha^l_j,\beta^l_{k}] \\
 +& \sum^{r_l/2}_{j=1}\sum^{r_l/2}_{k=1}  ((A^l)^TC^l)_{j,k} [\alpha^l_j ,\alpha^l_k]            +\sum^{r_l/2}_{j=1}\sum^{r_l/2}_{k=1}  ((B^l)^T D^l)_{j,k}[\beta^l_{j},\beta^l_{k}]
 \end{align*} This expression is homotopic to $\sum_{i=1}^{r_l/2} [\alpha_i^l,\beta_i^l] $, if
$$\begin{aligned}
	& (A^l)^TD^l +(-1)^{n/2}(C^l)^TB^l=1  \\
	& (A^l)^TC^l+(-1)^{n/2} (C^l)^TA^l=0\\
	& (B^l)^TD +(-1)^{n/2}(D^l)^T B^l =0\\
	& (A^l)^TC^l \text{ and } (B^l)^TD^l \text{ have diagonal entries in $\Lambda,$}
\end{aligned}$$
where $$\Lambda=\begin{cases} 2\IZ &\text{ if $n=2d$ and $d$ is odd and not $3$ or $7$} \\
\IZ & \text{ if $n=2d$ and $d$ is $3$ or $7$}\\
0 &\text{ if $n=2d$ and $d$ is even.}\end{cases}$$ Where the diagonal entries of $(A^l)^TC^l$ and $(B^l)^TD^l$ have to be in $\Lambda$ to kill the elements $[\alpha^l_i ,\alpha^l_i]$ and $[\beta^l_i ,\beta^l_i].$ These are exactly the conditions to be an automorphisms of $H(\IZ^{g_{n/2}})$ in $Q^{(-1)^{n/2}}(\IZ,\Lambda).$
Combining this with the condition in (\ref{mcgconsition1}) we see that the image of $\tilde H_*$ in $\Aut(\tilde H_*(N))$ is given by $$\Gamma_I \subset \prod^{\lfloor n/2\rfloor}_{k=1} Gl_{r_k}(\IZ).$$ Thus we proved that $$\tilde H_*:\pi_0(\aut_\partial (N)) \rightarrow \Aut(\tilde H_*(N),\langle-,-\rangle,Jq)$$ is surjective.
To show that the kernel is finite it suffices to check that $\pi_0(\aut_\partial (N))\rightarrow \pi_0(\aut_* (N))$ has finite kernel. This follows from the fact that $$\pi_1(in^*):\pi_1(\aut_*(N),\operatorname{id}_N)\otimes \IQ\rightarrow \pi_1(\map_*(\partial N, N),in)\otimes \IQ$$ is surjective as we will see in Remark $\ref{ratsur}.$

\end{proof}

In fact it suffices to know $\Gamma_I$ for our purposes, since the action of elements of the kernel of $\tilde H_*$ by conjugation is trivial up to homotopy.

\begin{lem}\label{gamma I action on universal cover}
Let $f$ represent an element of the kernel of $$\tilde H_*:\pi_0( \aut_\partial(N_I))\to \Gamma_I,$$ then $f^{-1}\circ g\circ f\simeq g$ for all $g\in \aut_\partial(N_I).$
\end{lem}

\begin{proof}
Note that if $[f]$ is in the kernel of $\tilde H_*$ then it is given by an element $\partial \alpha,$ where $\alpha\in \pi_1(\map_*(\partial N_I,N_I,in))\cong \pi_n(N_I).$ We represent $\alpha$ as a map $$\alpha(x,t):\partial N_I\times I\rightarrow N_I, \text{ such that } \alpha(x,0)=\alpha(x,1)=\id_{\partial N_I}.$$ Choosing a collar neighborhood of $\partial N_I$ allows us to identify $$N_I\xrightarrow{\cong}N_I\cup_{\id_{\partial N_I}}\partial N_I\times I.$$ Now we represent $f$ by the composite $$N_I\xrightarrow{\cong}N_I\cup_{\id_{\partial N_I}}\partial N_I\times I\xrightarrow{\id_{N_I}\cup \alpha} N_I.$$ We represent $f^{-1}$ similarly using $-\alpha$. If we now represent $f^{-1}\circ g\circ f$ by the composition $$\begin{aligned}
N_I\xrightarrow{\cong}N_I\cup_{\id_{\partial N_I}}\partial N_I\times I\xrightarrow{\id_{N_I}\cup \alpha} N_I \xrightarrow{\cong}N_I\cup_{\id_{\partial N_I}}\partial N_I\times I\\ \xrightarrow{g\cup \id_{\partial N_I\times I}}  N_I\cup_{\id_{\partial N_I}}\partial N_I\times I\xrightarrow{\id_{N_I}\cup -\alpha} N_I,
\end{aligned}$$ and see that it is homotopic to $$N_I\xrightarrow{\cong}N_I\cup_{\id_{\partial N_I}}\partial N_I\times I\cup_{\id_{\partial N_I}}\partial N_I\times I \xrightarrow{g\cup\alpha \cup -\alpha}N_I.$$ This is homotopic (rel. $\partial N_I$) to $g,$ since $$\partial N_I\times I\cup_{\id_{\partial N_I}}\partial N_I\times I \xrightarrow{\alpha \cup -\alpha}N_I$$ is homotopic to the inclusion of $\partial N_I\times I\cup_{\id_{\partial N_I}}\partial N_I\times I$ as a collar.

\end{proof}

Denote by $\diff_\partial(N)$ the space of self-diffeomorphism of $N$ fixing a collar neighborhood of the boundary point-wise with the Whitney $C^\infty$-topology. Let $J:\diff_\partial(N) \rightarrow\aut_\partial(N)$ be the canonical inclusion. To show homological stability for block diffemorphism the following fact about the mapping class group suffices.

\begin{prop}\label{mcg} 
\begin{itemize} \item[]
\item[$\operatorname{(1)}$] The map $\tilde H_*: \pi_0 (\diff_\partial(N))\rightarrow \Aut({\tilde H_*(N),\langle-,-\rangle,q})$ is surjective.
\item[$\operatorname{(2)}$] The image of $\pi_0(J):\pi_0 (\diff_\partial(N))\rightarrow \pi_0(\aut_\partial(N))$ has finite index.
\end{itemize} 
\end{prop}

\begin{proof}
The first part follows from \cite{MR561244} and \cite[Lemma 17]{MR0156354}. Kreck shows that all elements of $\Aut({\tilde H_{n/2}}(N),\langle-,-\rangle,q)$ can be realized as self-diffeomorphisms of $$(\#_{g_{n/2}}(S^{n/2}\times S^{n/2}))\smallsetminus \interior(D^n)$$ fixing the boundary pointwise. Wall shows that for manifolds $\natural_g (D^{q+1}\times S^p)$, where $3\leq p\leq q$ and $\natural_g$ denotes the $g$-fold boundary connected sum, all automorphisms of the homology are realized by diffeomorphisms. Hence it follows for manifolds $\#_{g_i} (S^{p_i}\times S^{q_i}).$ Since we can assume that a diffeomorphism fixes a disk up to isotopy, we get it in particular for $(\#_{g_i} (S^{p_i}\times S^{q_i}))\smallsetminus\interior(D^n).$ Using the diffeomorphisms above and extending them by the identity on the complement of the manifolds above the claim follows. The second part follows from the following commutative diagram
\begin{equation*}
\xymatrix{
0 \ar[r]& \pi_0 (\operatorname{S}\diff_\partial (N)  )\ar[r]\ar[d]  & \pi_0 (\diff_\partial (N)) \ar[r]^<<<<<{\tilde H_*} \ar[d]_{\pi_0(J)}& \Aut({\tilde H_*(N),\langle-,-\rangle, q}) \ar[r] \ar[d]^\cong & 0 \\
0 \ar[r]& \pi_0 (\operatorname{S}\aut_\partial (N)) \ar[r] & \pi_0 (\aut_\partial (N)) \ar[r]^<<<<<{\tilde H_*}& \Aut({\tilde H_*(N),\langle-,-\rangle , Jq}) \ar[r]& 0,  }
\end{equation*}
where $\pi_0 (\operatorname{S}\diff_\partial (N))$ and $\pi_0 (\operatorname{S}\aut_\partial (N))$ denote the kernels of the maps $\tilde H_*$ and the fact that $\pi_0 (\operatorname{S}\aut_\partial (N))$ is finite by Proposition \ref{hmcg}.

\end{proof}

\begin{rem}\label{mcglitremark}
There is a lot of literature on the groups of components of mapping spaces of closed manifolds in different categories. Highly connected even-dimensional manifolds are for example studied in \cite{MR561244} and \cite{MR0245031}. Products of spheres are studied in \cite{MR0248848,MR0253369,MR0250323}. Homotopy self-equivalences of manifolds and in particular of connected sums of products of spheres are treated in \cite{MR1340168}. 

\end{rem}

For later use we need the following Lemma.

\begin{lem}\label{H1 triv pi0} The groups 
$\pi_{0}\aut_\partial(N_I)$ and $\operatorname{Im}(\pi_0(J))$ are rationally perfect.
\end{lem}

\begin{proof} This follows from Lemma \ref{gammaIh1} and the fact that the groups are finite extension of $\Gamma_I.$
\end{proof}
\section{On the rational homotopy type of homotopy automorphisms}\label{On the rational homotopy type of homotopy automorphisms}

In the last section we determined the group $$\pi_I=\pi_1(B\aut_\partial (N_I),\id_{N_I})\cong \pi_0(\aut_\partial (N_I))$$ up to finite extensions. It acts on the simply-connected covering $X_I=B\aut_\partial (N_I)\univ$ by deck transformations. This section has two goals:
\begin{enumerate}
\item Describe the $\pi_I$-modules $H_*(X_I;\IQ)$ algebraically (Proposition \ref{iso Hce and H as piI-modules}).
\item Make sure the algebraic model is appropriate for showing homological stability using the results in Section \ref{van der kallens and charneys homological stab results} (Proposition \ref{derivation as schurfunctor}).
\end{enumerate}

All the results in this section are either contained in \cite{berg13} or straightforward generalizations.  

We assume some familiarity with Quillen's approach to rational homotopy theory \cite{MR0258031}, i.e. the functor $$\lambda:\operatorname{Top}_1\rightarrow \operatorname{dgL}_0$$ from the category of simply connected based topological spaces to the category of reduced differential graded (dg) Lie algebras. It induces an equivalence of homotopy categories, where the weak equivalences in $\operatorname{Top}_1$ are isomorphisms in rational homotopy groups and in $\operatorname{dgL}_0$ quasi isomorphisms. The homology of $\lambda(X)$ allows us to recover the rational homotopy groups of $X.$ More precisely, there is an isomorphism of graded Lie algebras $$H_*(\lambda(X))\cong \pi_*(\Omega X)\otimes \IQ,$$ where the Lie bracket on the right hand side is given by the Samelson product. The rational homology of $X$ is given by the Chevalley-Eilenberg homology of $\lambda(X),$ which we will explain later.

For a given simply connected space $X$ the value $\lambda(X)$ is in general very complicated and one considers dg Lie models instead. A \emph{dg Lie model} for a simply connected topological space $X$ is a free differential graded Lie algebra $(\IL(V),\partial),$ together with a quasi isomorphism $$(\IL(V),\partial) \xrightarrow{\simeq} \lambda(X).$$ When $\pi_*(\Omega X)\otimes \IQ$ is quasi isomorphic to $\lambda(X)$, the space $X$ is called \emph{coformal}.

\subsection{On a dg Lie model for the simply-connected covering of the homotopy automorphisms}

Since $N_I\simeq \bigvee_{i\in I} (S^{p_i}\vee S^{q_i})$, the free Lie algebra $\IL(s^{-1}\tilde H_*(N_I,\IQ))$ with trivial differential is a dg Lie model for $N_I,$ where $s^{-1}$ denotes the desuspension.
We are going to denote: $$\IL_I=\IL(s^{-1}\tilde H_*(N_I,\IQ)).$$ Recall that we denoted the homology classes represented by the inclusions $$\alpha_i:S^{p_i}\hookrightarrow N_I,\; \beta_i: S^{q_i}\hookrightarrow N_I,$$ by $a_i$ respectively $b_i$. Denote by $$\omega_I =\sum_{i\in I} -(-1)^{|a_i|}[ s^{-1}a_i, s^{-1}b_i ].$$ We model the inclusion of the boundary $in:\partial N_I\rightarrow N_I$ by $$\IL(\gamma)\rightarrow \IL_I\text{, }\gamma\mapsto \omega_I,$$ where $\IL(\gamma)$ is generated by a single generator of degree $n-2.$ 

Let $f:L\rightarrow K$ be a map of differential graded Lie algebras. We say that a degree $n$ linear map $\theta\in \hom_n(L,K)$, is an \emph{$f$-derivations} of degree $n$, if $$\theta [x,y]=[\theta(x),f(y)]+(-1)^{n|x|}[f(x),\theta(y)]\text{, for all }x,y\in L.$$ The $f$-derivations form a differential graded vector space $\der_f(L,K),$ with differential given by $$D(\theta)=d_K\circ \theta - (-1)^{|\theta |}\theta\circ d_L.$$ The \emph{derivations} of a differential graded Lie algebra $L$ is the special case $$\der(L)= \der_{\text{id}_L}(L,L).$$ We define a bracket for $\theta,\eta\in\der (L)$, by $$[\theta,\eta]=\theta \circ \eta -(-1)^{|\theta||\eta|}\eta\circ \theta,$$ which makes $\der (L)$ into a differential graded Lie algebra.

Denote by $\Der_\omega(\IL_I)$ the derivation Lie algebra annihilating $\omega_I,$ i.e. the kernel of the evaluation at $\omega_I$ $ev_{\omega_I}:\Der(\IL_I)\to \IL_I.$ The \emph{positive truncation} $L^+$ of a dg Lie algebra $L$ is given by $$  L_i^+ = \begin{cases} L_i&  \text{for } i\geq 2 \\
     \text{ker}(d_L:L_1\rightarrow L_0)&  \text{for } i=1 \\
    0 & \text{for } i\leq 0 \\
   \end{cases}$$ with its obvious differential and Lie bracket.

\begin{prop}[{Special case of \cite[Corollary 3.11]{berg13}}]\label{der isomorphism of gLie}
The simply-connected covering $B\aut_\partial (N_I)\univ$ is coformal and there is an isomorphism of graded Lie algebras $$\pi_*(\Omega B\aut_\partial (N_I)\univ)\otimes \IQ\cong\Der^+_\omega(\IL_I)$$
\end{prop}

\begin{proof} The manifolds $N_I$ are formal (in the sense of Sullivan's rational homotopy theory) and have trivial reduced rational cohomology rings since they are homotopy equivalent to wedges of spheres. Thus we can apply \cite[Corollary 3.11]{berg13}.
\end{proof}

\subsection{Derivations and the cyclic Lie operad}

In this section we collect the results from \cite{berg13} Section 6.1 and 6.2. Let $V$ be a graded finite dimensional rational vector space. By an inner product of degree $D$ we mean a degree $-D$ morphism $$\langle-,-\rangle:V\otimes V\rightarrow \IQ$$ that is non-degenerate in the sense that the adjoint $$V\to \hom(V,\IQ),\; v\mapsto \langle x,-\rangle,$$ is an isomorphism of graded vector spaces. The inner product is graded anti-symmetric if $$\langle x,x\rangle = -(-1)^{|x||y|} \langle y,x\rangle \text{ for all } x,y\in V.$$ 
Denote by $Sp^D$ the category with objects graded finite dimensional rational vector spaces $V$ together with a graded anti-symmetric inner product $\langle-,-\rangle_V$. The morphisms in $Sp^D$ are linear maps that respect the inner-product. For a morphisms $f:V\rightarrow W,$ there is a unique linear map $f^!$ such that $$\langle x,y \rangle_V=\langle f(x),f(y) \rangle_W\text{, for all }x,y\in V.$$ Since $f^!f=\id_V$ we get that morphisms in $Sp^D$ are injective.

Given a object $V$ of $Sp^D$, consider the inner product $\langle -,- \rangle_V$ as an element of $\hom(V^{\otimes 2},\IQ).$ We identify $V^{\otimes 2} \cong \hom(V^{\otimes 2},\IQ)$ using the inner product on $V^{\otimes 2}$ defined by $$\langle x\otimes y, x'\otimes y'\rangle = (-1)^{|x'||y|}\langle x,x'\rangle\langle y,y'\rangle.$$ Thus $\langle -,- \rangle_V$ gives rise to an element $\omega_V\in V^{\otimes 2}.$ The anti-symmetry of $\langle -,- \rangle $ implies that we can consider $\omega_V$ as an element $\omega_V \in \IL(V).$ When we choose a graded basis $\iota_1,...,\iota_r$ with dual basis $\iota_1^\#,...,\iota_r^\#$ for $V$ then $$\omega_V=\pm\frac{1}{2}\sum_i [\iota_i^\#,\iota_i].$$
\begin{example} We consider $s^{-1}H_I\otimes \IQ$ as an element of $Sp^{(n-2)}.$ Then $\omega_{s^{-1}H_I\otimes \IQ}$ is equal to $\omega_I$ up to sign.
\end{example}

A $Sp^D$-module in a category $\mathcal{V}$ is a functor $Sp^D\rightarrow \mathcal{V}.$ Our goal in this section is to show that we can describe $\der_\omega (\IL(-))$ as a $Sp^D$-module in a category of graded Lie algebras $gL$. Moreover we are going to see that this functor is in fact naturally equivalent to a Schur functor.

It is clear that $\IL(-):Sp^D\to gL$ defines a functor. Moreover using the adjoint $f^!$ of a morphisms $f:V\to W$ in $Sp^D,$ it follows that we can consider $\der(\IL(-)):Sp^D\to gL$ as a functor, where $\der(\IL(f))(\theta)$ for $\theta \in \der (\IL(V))$is given by the unique derivation defined by $$ \der(\IL(f))(\theta)(x)=\IL(f)\theta(f^!(x)) \text{ for }x \in W.$$
(see Proposition \cite[Proposition 6.1]{berg13}, where it is also shown that $\der(\IL(f))(\theta)$ is injective). Proposition 6.2 in \cite{berg13} now states that for a morphisms $f:V\to W$ in $Sp^D,$ the diagram \begin{displaymath}
\xymatrix{\der (\IL(V)) \ar[r]^{ev_{\omega_V}}\ar[d]_{\der(\IL(f))}  & \IL(V) \ar[d]^{\IL(f)} \\
\der (\IL(W)) \ar[r]^{ev_{\omega_W}} &\IL(W)}\end{displaymath} commutes. This implies in particular, that we get a functor $$\der_\omega (\IL(-)):Sp^D\to gL,$$ given by the kernel $\der_\omega (\IL(V))$ of the evaluation map $ev_{\omega_V}:\der(\IL(V))\rightarrow \IL(V)$ for $V\in Sp^D$.

We identify the $Sp^D$-module $\der(\IL(-))$ with the $Sp^D$-module $\IL(V)\otimes V,$ upon using the map $$\theta_{-,-}:\IL(V)\otimes V\to \der(\IL(-)),\; \theta_{\chi,x}(y)=\chi \langle x,y\rangle\text{, for $x\in V$ and $\theta\in \IL(V)$  }$$ (see \cite[Proposition 6.3]{berg13}). Under this identification the evaluation map becomes the map induced by the Lie bracket, i.e. \begin{equation*}
\xymatrix{
\IL(V)\otimes V \ar[rr]^-{[-,-]}\ar[d]_{\theta_{-,-}}  && \IL(V) \ar@{=}[d] \\
\der(\IL(V)) \ar[rr]^{ev_{\omega_V}} &&\IL(V)}
\end{equation*} commutes. Denote by $\g(V)$ the kernel of $[-,-]$ and observing that $[-,-]$ surjects onto the sub graded Lie algebra $\IL^{\geq 2}(V)$ of elements of bracket length $\geq 2,$ we get a commutative diagram of $Sp^D$-modules
\begin{equation}\label{gnat}
\xymatrix{
0\ar[r]&\g(V)\ar[r]\ar@{.>}[d]_\cong &\IL(V)\otimes V \ar[rr]^-{[-,-]}\ar[d]_{\theta_{-,-}}  && \IL^{\geq 2}(V) \ar@{=}[d]\ar[r] & 0\\
0\ar[r]&\der_{\omega}(\IL(V))\ar[r] &\der(\IL(V)) \ar[rr]^{ev_{\omega_V}} &&\IL^{\geq 2}(V)\ar[r] & 0. } \end{equation} 
Note that in the top row we do not use the inner product thus it defines in fact a functor $\g$ from the category of graded vector spaces.

Denote by $\Lie= \{\Lie(n)\}_{n\geq 0}$ the graded Lie operad and by $\Lie((n))$ the cyclic lie operad, i.e. the Lie operad in arity $n-1$ $\Lie(n-1)$ seen as a $\Sigma_{n}$-modules. Denote by $t=(123...n)\in \Sigma_n$ the cyclic permutation and denote by $t*\xi$ the action of $t$ on $\xi \in \Lie((n))$. There are short exact sequences of $\Sigma_n$-modules $$0\to \Lie((n)) \xrightarrow{\mu} \IQ[\Sigma_n]\otimes_{\Sigma_{n-1}}\Lie(n-1)\xrightarrow{\epsilon} \Lie(n)\rightarrow 0,$$ where $\mu(\xi)=\sum_i t^i\otimes t^{-i} * \xi$ and $\epsilon(\sigma\otimes \zeta)=\sigma[\zeta,x_n]$ (see \cite[Proposition 6.4]{berg13}).

Using the exact sequence we identify the rows of (\ref{gnat}) with $$ s^{-D} \bigoplus_{n\geq 2} \Lie((n))\otimes_{\Sigma_n} V^{\otimes n}\rightarrow s^{-D}\bigoplus_{n\geq 2}\Lie(n)\otimes_{\Sigma_n} V^{\otimes n} \rightarrow \bigoplus_{n\geq 2}\Lie(n) \otimes_{\Sigma_n} V^{\otimes n}.$$ Motivated by this we define $$\Lie((V))=s^{-D} \bigoplus_{n\geq 2} \Lie((n))\otimes_{\Sigma_n} V^{\otimes n}.$$ The Lie algebra structure on $\Lie((V))$ can be in fact be explicitly described using the cyclic operad structure of $\Lie((n))$ and the one on $V^{\otimes n}$ given by contractions, but we are not going to need it. We summarize the above as:
\begin{prop}[{\cite[Proposition 6.6.]{berg13}}]\label{derivation as schurfunctor} There is an isomorphism of $Sp^D$-modules in $gLie$ $$\Lie((-))\cong \der_\omega (\IL(-)).$$
\end{prop}

\begin{rem}
As a composition of the Schur functors $s^{-D}$ and the one given by the $\Sigma$-module $\{\Lie((n))\}_{n\geq1}$, we see that $\Lie((-))$ is also a Schur functor.
\end{rem}

\subsection{The action of the homotopy mapping class group}

To identify the action induced by deck transformations on $\Der^+_\omega(\IL_I)$ we begin by noting that the Lie algebras (both with Samelson product) $\pi_*(\Omega B\aut_\partial (N_I))\otimes \IQ$ and $\pi_*(\aut_\partial (N_I))\otimes \IQ$ are naturally isomorphic since $\aut_\partial (N_I)$ is a group-like monoid. Moreover the deck transformation action can be described in terms of the Samelson product, i.e. when we identify $\pi_*(B\aut_\partial (N_I))\cong\pi_{*-1}(\aut_\partial (N_I))$, the deck transformation action of $\pi_1(B\aut_\partial (N_I)\univ)$ corresponds to the action of $\pi_0(\aut_\partial (N_I))$ on $\pi_{*-1}(\aut_\partial (N_I))\otimes \IQ$ given by conjugation. The main tool to identify the action is the following theorem.

\begin{thm}[{\cite{lupton} stated as in \cite[Theorem 3.6]{berg13}}]\label{lsiso}
 Let $f:X\rightarrow Y$ be a map of simply connected CW-complexes with $X$ finite and $\phi_f:\IL_X\rightarrow \IL_Y$ a Quillen model. There are natural isomorphisms of sets 
\begin{equation*}
  \pi_k (\map_* (X,Y),f)\otimes \IQ \cong H_k(\der_{\phi_f}(\IL_X,\IL_Y)), \text{ for }k\geq 1,
\end{equation*} which are vector space isomorphisms for $k>1.$ In the case $X=Y$ and $f=$id$_X,$ there are isomorphisms of vectorspaces \begin{equation*}\label{natiso}
  \pi_k (\aut_* (X),\text{id}_X)\otimes \IQ \cong H_k(\der(\IL_X)), \text{ for }k\geq 1,
\end{equation*} and the Samelson product corresponds to the Lie bracket.
\end{thm}

\begin{prop}[{Compare \cite[Proposition 5.5]{berg13}}]\label{der iso comp with action} There is a $\pi_0(\aut_\partial(N_I))$-equivariant isomorphism of graded Lie algebras
$$\pi_{*}^+(\aut_\partial (N_I))\otimes \IQ\cong\Der^+_\omega(\IL_I),$$ where the action on the right hand side is through the canonical action of $\Gamma_I$ on $H_I$

\end{prop}

\begin{proof}

We are going to study the long exact sequence of rational homotopy groups of the fibration
$$\aut_\partial(N_I) \rightarrow \aut_*(N_I)\rightarrow \map_*(\partial N_I,N_I).$$ Denote by $\varphi:\IL(\omega)\rightarrow \IL_I$ the inclusion of the sub graded Lie algebra of $\IL_I$ generated by $\omega_I.$  Using Theorem \ref{lsiso} we see that the map $$\pi_k(\aut_*(N_I))\otimes \IQ\rightarrow \pi_k(\map_*(\partial N_I,N_I))\otimes \IQ$$ is given by \begin{equation*}
\der(\IL_{I})_k  \xrightarrow{\varphi^*_k}\der_{\varphi}(\IL (\omega_I) ,\IL_{I})_k,
\end{equation*}where $\varphi^*_k$ is the restriction to $\IL(\omega_I).$ Note that $\der_{\varphi}(\IL (\omega_I) ,\IL_{I})\cong s^{(n-2)}\IL_I.$ Under this identifications the map $\varphi^*$ becomes the evaluation map which is surjective as discussed for the diagram (\ref{gnat}). Hence the long exact sequence of rational homotopy groups splits as \begin{equation*}
0\rightarrow \der_\omega(\IL_{I})_*\rightarrow \der(\IL_{I})_*  \xrightarrow{\varphi^*_k}\der_{\varphi}(\IL (\omega_I) ,\IL_{I})_*\rightarrow 0,
\end{equation*} where we use that $\der_\omega(\IL_{I})$ is the kernel of the evaluation map. The resulting isomorphism $$\pi_*^+(\aut_{\partial}({N_I}),\id_{N_I})\otimes \IQ\cong \der^+_{{\omega} }(\IL_{I})$$ is in fact an isomorphism of graded Lie algebras. Indeed, since the inclusion $\aut_\partial ({N_I})\rightarrow  \aut_*({N_I})$ is a map of topological monoids, the induced maps on rational homotopy group respect the Samelson product and $\der^+(\IL_{I})\cong\pi^+_*(\aut(N_I))\otimes \IQ$ is an isomorphism of Lie algebras. Hence we can calculate the Samelson product of $\pi^+_*(\aut_\partial(N_I))$ in $\der^+(\IL_{I}).$

Now let $f,g\in \aut_*({N_I})$. The action of $[f]\in\pi_0(\aut_*({N_I}))$ on $\pi_k(\aut_*({N_I}))$ is induced by pointwise conjugation $g\mapsto fgf^{-1},$ where $f^{-1}$ is some choice of homotopy inverse. Let $\phi_f$ be a Quillen model for $f$ and $\theta\in \der ({\IL_I})_k.$ The action of $[f]$ on $\der(\IL_{I})_k$ is given by 
\begin{equation*}
\theta\mapsto \phi_f\circ \theta\circ \phi_f^{-1},
\end{equation*} by the naturality of the identification $\pi_k(\aut_*({N_I}))\otimes \IQ \cong \der(\IL_{I})_k.$ 
For a homotopy self-equivalence $f$ consider the induced map $f_*\in \Aut (\tilde{H_*}(N_I))$. The map $\IL(s^{-1}(f_*\otimes \IQ))$ is in fact a Lie model for $f$, which shows that we can identify the conjugation action with the induced action of $\Aut (\tilde{H_*}(N_I))$ on $\der(\IL_I)_k$.\\
Using that $\der^+_\omega(\IL_{I})_k\rightarrow \der^+(\IL_{I})_k$ is injective, we calculate the conjugation action of $\pi_0(\aut_{\partial}({N_I}))$ on $\pi_k(\aut_{\partial}({N_I}),\id_{N_I})$ in terms of $\der_{{\omega} }(\IL_{I})_k$. Let $f$ be an element of $\aut_{\partial}({N_I})$, it is in particular also an element of $\aut_*({N_I})$ and we know that its homotopy class $[f]$ in $\pi_0 (\aut_*({N_I}))$ gives us an element in $\Gamma_I$. Considering $\theta\in \der_{{\omega}}(\IL_{I})_k$ as an element in $\der(\IL_{I})_k$ we see that $[f]$ acts by the action induced by $f_*\in \Gamma_I$ on $\tilde H_*(N_I).$
\end{proof}

\begin{rem}\label{ratsur} Observe that we discussed in the proof that the map $$\pi_1(\aut_* (N_I)) \otimes \IQ\rightarrow \pi_1(\map_*(\partial N_I,N_I))\otimes \IQ$$ is surjective hence the kernel of $\pi_0(\aut_\partial(N_I))\rightarrow \Gamma_I$ is finite.
\end{rem}

We need to identify the maps induced by the stabilization map on rational homotopy groups. Given an element of $\theta\in \der^+_{\omega}(\IL_I)$ we define an element $\theta'\in\der^+_{\omega}(\IL_{I'}),$ by letting $\theta'=\theta$ on generators $\iota_i,\kappa_i,$ where $i\in I$ and $\theta(\iota_{i'})=\theta(\kappa_{i'})=0.$ Using that $\IL_{I'}$ is free we get a derivation $\theta',$ which is indeed an element of $\der^+_{\omega}(\IL_{I'}),$ since $\omega_{I'}=\omega_{I}+(-1)^{|\iota_{i'}|}[\iota_{i'},\kappa_{i'}].$ We refer to this map again as stabilization map.

\begin{prop}\label{der iso compatible with stab maps}
The isomorphism $$\pi_*(\aut_\partial (N_I)\univ)\otimes \IQ\cong\Der^+_\omega(\IL_I)$$ is compatible with the stabilization maps. 
\end{prop}

\begin{proof} This works exactly as in \cite[Proposition 7.7]{berg13}.
\end{proof}

Ultimately we describe the rational homology $H_*(B\aut_\partial(N_I)\univ),\IQ)$ as $\pi_I$-modules. The link between a dg Lie model and the rational homology of a space is given by the Chevalley-Eilenberg homology. The Chevalley-Eilenberg complex of a dg Lie algebra $L$ with differential $d_L$ is the chain complex $C_*^{CE}(L) = \Lambda_* sL$ with differential $\delta^{CE}=\delta_0^{CE} + \delta_1^{CE}$, where  $s$ denotes the suspension and $\Lambda_*$ the free graded commutative algebra. In the simplest cases the differentials are given by\begin{equation*}
\delta_0^{CE}(sx) =-s d_L x
\end{equation*}\begin{equation*}
\delta_1^{CE}(sx_1 \wedge sx_2) =(-1)^{|x_1|}s[x_1,x_2],
\end{equation*}
where $x,x_1,x_2\in L$.

Quillen showed that for a dg Lie model $L_X$ of a space $X$ that the Chevalley-Eilenberg homology gives the rational homology groups of $X,$ i.e. that $H^{CE}_*(L_X)\cong H_*(X;\IQ).$ 

Grade the Chevalley-Eilenberg chains by word length, i.e. let $(\Lambda^{p}(L))_q$ be the elements of word length $p$ and degree $q$. Denote by $H_{p,q}^{CE} (L)$ the homology of the chain complex $$...\to (\Lambda^{p+1}(L))_q\xrightarrow{\delta_1}(\Lambda^{p}(L))_q\xrightarrow{\delta_1}(\Lambda^{p-1}(L))_q\to ...$$
The Quillen Spectral Sequence is the spectral sequence coming from the filtration by word length. In case that the dg Lie algebra $L_X$ is a model for a space $X,$ we can identify the $E^2$-page with $$E^2(L)_{p,q} = H_{p,q}^{CE} (L_X)\cong H_{p,q}^{CE} (\pi_*(\Omega X)\otimes \IQ)\Rightarrow H_*^{CE}(L_X)\cong H_*(X;\IQ).$$  

The Quillen spectral sequence collapses on the $E^2$-page for coformal spaces, hence we get isomorphisms $$H_r(B\aut_\partial(N_I)\univ),\IQ)\cong \bigoplus_{p+q=r} H_{p,q}^{CE} (\pi_*(\Omega B\aut_\partial(N_I)\univ)\otimes \IQ).$$ The Quillen spectral sequence is in fact natural with respect to unbased maps of simply connected spaces (\cite[Proposition 2.1.]{berg13}). Since $\pi_I$ is rationally perfect (see Lemma \ref{H1 triv pi0}), we do not have any extension problems, hence the isomorphism above is in fact an isomorphism of $\pi_I$-modules (see \cite[Proposition 2.3]{berg13}). 

\begin{prop}\label{iso Hce and H as piI-modules}
There are isomorphisms of $\pi_I$-modules $$H^{CE}_r(\der_\omega(\IL_I))\cong H_r(B\aut_\partial(N_I)\univ;\IQ)$$ compatible with the stabilization maps
\end{prop}

\begin{proof}We use the isomorphism of $\pi_I$-modules
$$H_r(B\aut_\partial(N_I)\univ),\IQ)\cong \bigoplus_{p+q=r} H_{p,q}^{CE} (\pi_*(\Omega B\aut_\partial(N_I)\univ)\otimes \IQ)$$ constructed above. By Proposition \ref{der iso comp with action}, we identify $$H_{p,q}^{CE} (\pi_*(\Omega B\aut_\partial(N_I)\univ)\otimes \IQ)\cong H_{p,q}^{CE} (\der_\omega (\IL_I))$$ as $\pi_I$-modules. This in turn gives the isomorphism in the claim.

The compatibility with the stabilization maps follows from Proposition \ref{der iso compatible with stab maps}.
\end{proof}

\section{Homological Stability}

\subsection{An algebraic homological stability result}

Recall the graded hyperbolic modules $H_I\cong \tilde H_*(N_I)$ from Section \ref{automorphisms of hyperbolic modules}, which have the groups $\Gamma_I$ as their automorphism groups. Recall that we denoted by $\lambda_I$ the $\Gamma_I$-module $H_I$ with standard action which we considered as an object $$((H_I)_1,...,(H_I)_{n-1})\in \Mod(\IZ)^{n-1}.$$
Moreover recall that we defined
\begin{displaymath}
	\begin{aligned}
		&g_p=\begin{cases}\operatorname{rank} ((H_I)_p) &\text{if } 2p<n \\  \operatorname{rank} ((H_I)_p)/2 &\text{if } 2p=n.\end{cases}\\
		&H_{I'}=H_I\oplus \IZ[p]\oplus \hom(\IZ[p],\IZ[n]) &&\text{for fixed $p<\lfloor n/2 \rfloor$}  \\
		&I_{p,n-p}:H_I\rightarrow H_{I'} &&\text{the upper inclusion}\\
		&\mathfrak{g}_I=\Derpw(\IL_I)=\Derpw(\IL(s^{-1}H_I\otimes \IQ)).
	\end{aligned}
\end{displaymath}

We begin by showing:

\begin{prop}\label{polynomial functor degree}
Let $n>3$. For all $l\geq 0$ there are polynomial functors $$\mathscr{C}_l:\Mod(\IZ)^{n-1}\rightarrow \Vect(\IQ)$$ of degree $\leq l/2$ and isomorphisms of $\Gamma_I$-modules
$$\mathscr{C}_l(\lambda_I)\cong C_l^{CE}(\mathfrak{g}_I)$$ compatible with the maps induced by inclusions.
\end{prop}

\begin{proof}

In Proposition \ref{derivation as schurfunctor} we described the derivations $\der_\omega(\IL_I)$ as a Schur functor $$\Lie((-)):Sp^{n-2}\to gLie,\; V\mapsto s^{-n+2}\bigoplus_{k\geq 2} \Lie((k))\otimes_{\Sigma_k}V^{\otimes k},$$ that extended to the category of graded vector spaces $Vect_*(\IQ)$. Consider the inclusion $$\mathcal{I}:\prod_{i=0}^{n-2} Vect(\IQ)\rightarrow Vect_*(\IQ), \; (V_i)_{i=0}^{n-2}\mapsto \bigoplus_{i=0}^{n-2} V_i[i].$$ The composition $\Lie((-))\circ \mathcal{I}$ is a Schur multifunctor, which we are denoting by $ \widetilde {\mathscr{U}},$ with $$\widetilde {\mathscr{U}}(\mu)=s^{[1m_1+2m_2+...+(n-2)m_{n-2}-n+2]}\Lie((|\mu|)),$$ where $\mu=(m_0,m_1,...,m_{n-2}),$ given the $\Sigma_\mu$-module structure induced by the inclusion $\Sigma_\mu\subset \Sigma_{|\mu|}.$ Thus the positive degree derivations are given by the Schur multifunctor $$\mathscr{U}:\prod_{i=0}^{n-2} Vect(\IQ)\to gLie,$$ where $\mathscr{U}(\mu)=\widetilde{\mathscr{U}}(\mu),$ when $$1m_1+2m_2+...+(n-2)m_{n-2}-n+2\geq 1\Leftrightarrow \sum_{i=1}^{n-2} m_i\frac{i}{n-1}\geq 1,$$ and $0$ otherwise. The Chevalley-Eilenberg chains are given by the Schur functor $\Lambda$ with $\Lambda(r)$ the trivial $\Sigma_r$-module concentrated in degree $r$. 
The composition $\widetilde{\mathscr{C}}=\Lambda\circ\mathscr{U}$ is now given by (using (\ref{schurcomp})) $$\widetilde{\mathscr{C}}_r(\mu)\cong  \bigoplus_r  \Lambda(r) \otimes_{\Sigma_r} \bigoplus\operatorname{Ind}_{\Sigma_{\mu_1}\times...\times \Sigma_{\mu_r}}^{\Sigma_\mu}\mathscr{U}(\mu_1)\otimes ... \otimes \mathscr{U}(\mu_r),   $$where the second sum runs over all $r$-tuples $(\mu_1,...,\mu_r),$ such that $\Sigma_{i=1}^r \mu_i=\mu$. For a fixed r we get that the summand corresponding to $(\mu_1,...,\mu_r),$ where $\mu_s=(m_{1,s},...m_{r,s})$, is only non-zero, if $\sum_{i=1}^{n-2} (m_{i,s} \frac{i}{n-1}) \geq 1 $ for all $s=1,...,r$. That implies that
\begin{equation}\label{r}
\sum_{i=1}^{n-2} m_i \frac{i}{n-1}= \sum_{s=1}^r (\sum_{i=1}^{n-2} m_{i,s} \frac{i}{n-1})\geq r.
\end{equation}
If the summand is non-zero it is of degree
\begin{equation*}
\begin{aligned}
r+ \sum_{s=1}^r \left(\left(\sum_{i=1}^{n-2} m_{i,s} i\right)-n+2\right) &=r+ \left(\sum_{i=1}^{n-2} m_ii\right) -nr + 2r\\
&=\left(\sum_{i=1}^{n-2} m_ii\right) + r(3-n)\\
&\geq\left(\sum_{i=1}^{n-2} m_ii\right)+  \left(\sum_{i=1}^{n-2} m_i \frac{i}{n-1}\right)(3-n) \\
&=\frac{2}{n-1}\left(\sum_{i=1}^{n-2} m_ii\right),
\end{aligned}
\end{equation*} where we used that $3-n$ is negative and (\ref{r}). This implies now that the Chevalley-Eilenberg $l$-chains are a Schur multifunctor $\widetilde{\mathscr{C}}_l,$ where $\widetilde{\mathscr{C}}_l(\mu)$ vanishes for 
$$l<\frac{2}{n-1}\left(\sum_{i=1}^{n-2} m_ii\right)\leq 2\sum_{i=1}^{n-2}m_i=2|\mu|.$$
Hence it is of degree $\leq l/2.$
The functor $\mathscr{C}_l$ in the statement is given by the pre-composition with the additive functor $$-\otimes \IQ: \Mod(\IZ)^{n-1}\to \prod_{i=0}^{n-2} Vect(\IQ).$$ The compatibility with the action follows from the fact the functor lifts to $$Q^n_+(\IZ,\Lambda)\rightarrow Sp^{(n-2)}, \; M\mapsto s^{-1}(M\otimes \IQ).$$

\end{proof}

As an immediate consequence of Proposition \ref{polynomial functor degree} and Proposition \ref{general algebraic homological stability result} we get:

\begin{cor}\label{homstab cor}
The stabilization map
$$H_k(\Gamma_I, C^{CE}_l (\mathfrak{g}_{I})\rightarrow H_k(\Gamma_{I'},C^{CE}_l (\mathfrak{g}_{I'}))$$
is an isomorphism for $g_p> 2k+2l+2$ when $2p\neq n$ and $g_p> 2k+2l+4$ if $2p=n$ and an epimorphism for $g_p\geq 2k+2l+2$ respectively $g_p\geq 2k+2l+4.$
\end{cor}

Before we proof the main proposition of this section, i.e. deduce homological stability for the Chevalley-Eilenberg homology from the stability for the chains, we need the following observation about the Chevalley-Eilenberg chains.
\begin{lem}\label{hyperhom lemma}
There exists a chain homotopy equivalences $C^{CE}_* (\mathfrak{g}_I)\xrightarrow{\simeq}H^{CE}_* (\mathfrak{g}_I)$ sending cycles $z\mapsto[z]$ such that
\begin{equation*}
\xymatrix{
C^{CE}_* (\mathfrak{g}_I) \ar[rr]^{\sigma_*}\ar[d]_{\simeq}  && C^{CE}_* (\mathfrak{g}_{I'}) \ar[d]_{\simeq} \\
H^{CE}_* (\mathfrak{g}_I) \ar[rr]^{\sigma_*} && H^{CE}_* (\mathfrak{g}_{I'})}
\end{equation*}
commutes up to chain homotopy of $\IQ[\Gamma_I]$-chain complexes.
\end{lem}
\begin{proof}
This is is true for all degree-wise finite dimensional $\IQ[G]$-chain complexes for $G$ rationally perfect groups by \cite[Lemma B.1]{berg13} and \cite[Proposition B.5]{berg13}. The groups $\Gamma_I$ are rationally perfect (see Lemma \ref{gammaIh1}) and the the $C^{CE}_*(\mathfrak{g}_I)$ are degree-wise finite dimensional, since the $\g_I$ are. 
\end{proof}

\begin{prop}\label{algebraic homological stability result}
The stabilization map
$$H_k(\Gamma_I, H^{CE}_l (\mathfrak{g}_I))\rightarrow H_k(\Gamma_{I'},H^{CE}_l (\mathfrak{g}_{I'}))$$
is an isomorphism for $g_p> 2k+2l+2$ when $2p\neq n$ and $g_p> 2k+2l+4$ if $2p=n$ and an epimorphism for $g_p\geq 2k+2l+2$ respectively $g_p\geq 2k+2l+4.$
\end{prop}

\begin{proof}
Consider the first hyperhomology spectral sequence with $E^1$-page:
$$ E^1_{k,l}(I)=H_l((\Gamma_I; C^{CE}_k (\mathfrak{g}_I)))\Rightarrow \IH_{k+l}(\Gamma_I; C^{CE}_* (\mathfrak{g}_I)).$$ By Corollary \ref{homstab cor} $E^1_{k,l}(I)\rightarrow E^1_{k,l}(I')$ is an isomorphism for $$g_p>\begin{cases} 2k+2l+4 \geq k+2l+4 & \text{ if } p=n/2  \\ 2k+2l+2 \geq k+2l+2 & \text{ otherwise} \end{cases}$$ and an epimorphism for $"\geq"$. By the spectral sequence comparison theorem we get that the map
$$\IH_i(\Gamma_I, C^{CE}_* (\mathfrak{g}_I))\rightarrow \IH_i(\Gamma_{I'},C^{CE}_* (\mathfrak{g}_{I'}))$$ induced by the stabilization map
is an isomorphism for $g_p>2i+2$ when $2p\neq n$ and $g_p>2i+4$ if $2p=n$ and an epimorphism for $"\geq"$.
Upon using Lemma \ref{hyperhom lemma} and the chain homotopy invariance of hyperhomology we get that the map
$$\IH_i(\Gamma_I, C^{CE}_* (\mathfrak{g}_I))\rightarrow \IH_i(\Gamma_{I'},C^{CE}_* (\mathfrak{g}_{I'}))$$ induced by the stabilization map
is an isomorphism and epimorphism in the same rage as above. Ultimately we use the natural splitting for hyperhomology groups with coefficients in a chain complex with trivial differential
\begin{equation*}
\xymatrix{
\IH_i(\Gamma_I;H^{CE}_* (\mathfrak{g}_I)) \ar[r]^-{\sigma_i}\ar[d]_{\cong}  & \IH_i(\Gamma_{I'}; H^{CE}_* (\mathfrak{g}_{I'})) \ar[d]_{\cong} \\
\bigoplus_{k+l=i}H_k(\Gamma_I;H^{CE}_l (\mathfrak{g}_I)) \ar[r]^-{\sigma_{k,l}} & \bigoplus_{k+l=i} H_k(\Gamma_{I'};H^{CE}_l (\mathfrak{g}_{I'})).}
\end{equation*}
Hence we see that the maps $\sigma_{k,l}$ are isomorphisms and epimorphisms in the range in the statement of Proposition \ref{algebraic homological stability result}.

\end{proof}
\subsection{Homological stability for homotopy automorphisms}

The first main result of this article now easily follows from the previous results.

\begin{THMA}
The map $$H_i(B\aut_\partial(N_I);\IQ)\rightarrow H_i(B\aut_\partial(N_{I'});\IQ)$$ induced by the stabilization map
is an isomorphism for $g_p> 2i+2$ when $2d\neq n$ and $g_p> 2i+4$ if $2d=n$ and an epimorphism for $g_p\geq 2i+2$ respectively $g_p\geq 2i+4.$
\end{THMA}

\begin{proof} We begin by observing that  $$H_k(\Gamma_I, H^{CE}_l (\mathfrak{g}_I))\cong H_k(\pi_I, H^{CE}_l (\mathfrak{g}_I)),$$ since the action of $\pi_I$ is through $\tilde H :\pi_I\to \Gamma_I$ and the kernel of $\tilde H_*$ is finite (Propositions \ref{hmcg} and \ref{der iso comp with action}).
The result now follows from Proposition \ref{algebraic homological stability result} combined with Proposition \ref{iso Hce and H as piI-modules} upon using the spectral sequence comparison theorem.
\end{proof}

\subsection{Homological stability for block diffeomorphisms}
  
Denote by $\baut_\partial(X)$ the $\Delta$-monoid of block homotopy equivalences, with $k$-simplices face preserving homotopy equivalences $$\varphi:\Delta^k\times X\rightarrow \Delta^k\times X,$$ s.t. $\varphi|\Delta^k\times \partial X$ is the identity. The block diffeomorphisms $\bdiff_\partial (X)$ are the sub-$\Delta$-group with $k$-simplices, face preserving diffeomorphisms $$\varphi:\Delta^k\times X\rightarrow \Delta^k\times X,$$ s.t.$\varphi$ is the identity on a neighborhood of $\Delta^k\times \partial X.$ We do not distinguish between $\Delta$-objects and their realizations. Denote the inclusion $\bdiff_\partial(X)\hookrightarrow\baut_\partial(X)$ by $\tilde J$. The inclusion $$\aut_\partial (X) \hookrightarrow \baut_\partial (X)$$ is a homotopy equivalences and hence we are going to consider then as identified. The block diffeomorphisms $\bdiff_\partial (X)$ and the diffeomorphisms $\diff_\partial (X)$ with the Whitney C$^\infty$-topology on the other hand are not homotopy equivalent - the difference is related to algebraic K-theory (see \cite{MR1818774}). The homogeneous space $\baut_\partial(X)/\bdiff_\partial(X)$ is by definition the homotopy fiber of the map $\tilde J:B\bdiff_\partial (X)\rightarrow B\baut_\partial (X).$ It is related to Surgery theory as we explain now.

Let $X$ be a smooth manifolds of dimension $\geq 5$ with boundary $\partial X.$ Quinn \cite{quinn} shows that there is a quasi-fibration of Kan $\Delta$-sets $$\mathcal{S}^{G/O}_\partial(X)\rightarrow \map_*(X/\partial X,G/O)\rightarrow \mathbb{L}(X) $$ and that its homotopy exact sequence is the surgery exact sequence. $\mathcal{S}^{G/O}_\partial(X)$ is the realization of a $\Delta$-set with $k$-simplexes pairs $(W,f),$ where $W$ is a smooth $(k+3)$-ad and $f:W\rightarrow \Delta^k\times X$ is a face preserving homotopy equivalence, such that $f$ restricts to a diffeomorphisms $f|\partial_{k+1}W:\partial_{k+1}W\rightarrow \Delta^k\times \partial X.$ There is a map $$\baut_\partial(X)/\bdiff_\partial(X)\rightarrow \mathcal{S}^{G/O}_\partial(X), $$ which by the h-cobordism theorem induces a weak homotopy equivalence $$\baut_\partial(X)/\bdiff_\partial(X)_{(1)}\simeq_{w.e.}\mathcal{S}^{G/O}_\partial(X)_{(1)}$$ of the identity components (see \cite[Section 3.2.]{berg12}). Now assume that $X$ is simply connected. Since $$G/O\simeq_{\IQ}BO \simeq_{\IQ} \prod_{i\geq 1} K(\IQ,4i),$$ we understand the rational homotopy groups $$\pi_*(\map_*(X/\partial X,G/O))\otimes \IQ \cong H^*(X,\partial X;\IQ)\otimes \pi_*(G/O).$$ Note that if $X$ is simply connected $$\pi_i(\mathbb{L}(X))\otimes \IQ\cong L_{\operatorname{dim}(X)+i}(X) \cong\begin{cases}\IQ  &\text{if} \operatorname{dim}(X) + i \equiv 0 \text{ mod }4\\
0  &\text {otherwise.}
\end{cases}$$

We now specialize to $N_I.$
\begin{lem}[{\cite[Lemma 3.5.]{berg12}}]\label{plumbinglemma} The surgery obstruction map induces an isomorphism
$$H^{n}(N_I,\partial N_I;\IQ)\otimes \pi_{n+k}(G/O)\rightarrow L_{n+k}(\IZ)\otimes\IQ$$ for $n+k\equiv 0 \text{ mod } 4.$
\end{lem}

\begin{proof}
Consider the smooth and topological surgery exact sequences:
\begin{displaymath}
\xymatrix{...\ar[r] &N_\partial^{G/O} (N_I\times D^k)\otimes \IQ \ar[d] \ar[r] & L_{n+k}(\IZ)\otimes \IQ\ar@{=}[d] \ar[r]&...\\
...\ar[r] & N_\partial^{G/Top} (N_I\times D^k)\otimes \IQ\ar[r] & L_{n+k}(\IZ)\otimes \IQ\ar[r]&...}  
\end{displaymath}
The left hand vertical map is an isomorphism since $\pi_i(Top/O)$ is finite (see e.g. \cite{MR0645390}). Milnor's Plumbing construction ensures that for $k+n$ even there is an element in $N_\partial^{G/Top} (N_I\times D^k)$ with non-trivial surgery obstruction. Since $$N_\partial^{G/Top} (N_I\times D^k)\cong \pi_k(map_* (N_I/\partial N_I,G/Top) \text{ and }$$
$$\pi_k(map_* (N_I/\partial N_I,G/O)\otimes \IQ\cong H^{n}(N_I,\partial N_I;\IQ)\otimes \pi_{n+k}(G/O)), \text{ for $n+k\equiv 0 \text{ mod } 4,$}$$ this implies the claim because both sides are just one dimensional rational vectorspaces.
\end{proof}

This now implies that we have a natural isomorphism
\begin{equation}\label{naturalmapandH^*}
\pi_k(\mathcal{S}^{G/O}_\partial(N_I))\otimes \IQ\cong H^n(N_I;\IQ)\otimes \pi_{n+k}(G/O) \text { for } *>0
 \end{equation}
(compare \cite[Corollary 4.6.]{berg13}).

Recall that $J_0\pi_0(\diff_\partial (N_I))$ has finite index in $\pi_0(\aut_\partial (N_I))$ (Proposition \ref{mcg}). By Cerf's pseudo-isotopy theorem thus also $\tilde J_1\pi_1(B\bdiff_\partial (N_I))$ in $\pi_1(B\aut_\partial (N_I)).$ Denote by $\bar B\aut_\partial (N_I)$ the finite cover corresponding to $\operatorname{Image}(\tilde J_1).$ Note that it has the same (higher) rational homotopy groups. By construction $\tilde J$ lifts to a map $B\bdiff_\partial (N_I)\rightarrow \bar B\aut_\partial (N_I).$ Instead of $\baut_\partial(N_I)/\bdiff_\partial(N_I)$ we consider $$\mathcal{F}_I=\operatorname{hofib}(B\bdiff_\partial (N_I)\rightarrow \bar B\aut_\partial (N_I)).$$ We reduced the problem of showing rational homological stability for the block diffeomorphisms to the study of the Serre spectral sequence of the homotopy fibration above. The only missing ingredient is now to understand the rational homology groups of $\mathcal{F}_I$ as $\bar\pi_I =\pi_1(\bar B\aut_\partial (N_I))$-modules. Observe that by Proposition \ref{mcg} there is a surjection $$\pi_1(\bar B\aut_\partial (N_I))\rightarrow \Gamma_I$$ induced by $\tilde H_*.$

Denote by $\eta: S^{G/O}_\partial(X) \rightarrow \pi_0\map_*(X/\partial X,G/O)$ the normal invariant and by $\sigma:\pi_0\map_*(X/\partial X,G/O) \rightarrow L_{\operatorname{dim}(X)}(X)$ the surgery obstruction. Using the surgery exact sequence, we see that for $n$ odd $\pi_1(\mathcal{F}_I)$ is abelian, since it is a subgroup of the abelian group $[\Sigma (N_I/\partial N_I),G/O]_*$. For $n$ even it is a finite extension of the abelian group $[\Sigma (N_I/\partial N_I),G/O]_*$ by a finite cyclic group (in case $L_{n+2}(\IZ)\cong \IZ$, the proof of Lemma \ref{plumbinglemma} makes sure that the map to $L_{n+2}(\IZ)$ is non zero and hence the kernel of $\sigma$ is a finite cyclic group). Denote by $$\pi_k^{ab}(\mathcal{F}_I)=\begin{cases} \pi_1(\mathcal{F}_I)/ \operatorname{Image}(L_{n+2}(\IZ)\rightarrow \pi_1(\mathcal{F}_I)) & \text{ if } k=1 \\ \pi_k(\mathcal{F}_I) & \text{ if } k>1.\end{cases} $$

\begin{prop}There are isomorphisms of $\bar\pi_I$-modules compatible with the stabilization maps
\begin{enumerate} 
	\item[$\operatorname{(1)}$] $\pi^{ab}_k(\mathcal{F}_I)\otimes \IQ\cong (\tilde H_{*}(N_I,\IQ) \otimes \pi_*(G/O))_{k}$, where $|a\otimes \alpha|=|\alpha|-|a|$, $k\geq 1$
	\item[$\operatorname{(2)}$] $H_*(\mathcal{F}_I,\IQ) \cong \Lambda(\pi^{ab}_*(\mathcal{F}_I)\otimes \IQ)$,
\end{enumerate} where the actions on the left hand side are induced by the standard actions of $\Gamma_I$ on $\tilde H_*(N_I).$ 
\end{prop}

\begin{proof} Compare \cite[p.26 and Theorem 3.6]{berg12} and \cite[Proposition 7.15.]{berg13}. Observe that the rationalization $(\mathcal{F}_I)_{\IQ}$ has rational homotopy groups $\pi^{ab}_k(\mathcal{F}_I)\otimes \IQ.$
Consider the splitting of the homotopy exact sequence of the surgery fibration as
$$0\rightarrow L_{n+k+1}(\IZ)/\operatorname{Image}(\sigma)\rightarrow \pi_k (\mathcal{S}_\partial^{G/O}(N_I)) \rightarrow \operatorname{Image}(\eta)\rightarrow 0.$$ By Lemma \ref{plumbinglemma} we get $$L_{n+k+1}(\IZ)/\operatorname{Image}(\sigma)\otimes \IQ\cong 0\text{ and }\operatorname{Image}(\eta)\otimes \IQ \cong (\tilde H^*(N_I,\IQ) \otimes \pi_*(G/O))_{k}.$$ Using the isomorphism $$\tilde H_*(N_I;\IQ)\cong \hom_\IQ(\tilde H_*(N_I;\IQ);\IQ)\cong \tilde H^*(N_I;\IQ)$$  we get the isomorphism $(1).$ We see that the action on the right hand side is induced by the standard action of $\Gamma_I$ as follows: Use the identification $$\pi_k (\mathcal{S}^{G/O}_\partial(X))\cong S_\partial^{G/O}( N_I\times D^k).$$ An element of $S_\partial^{G/O}( N_I\times D^k)$ is represented by a manifold $(X,\partial X)$ together with a homotopy equivalence $f:X\rightarrow N_I\times  D^k,$ such that $f|\partial X:\partial X\rightarrow\partial (N_I\times D^k)$ is a diffeomorphism. The action of a $$[\phi]\in\pi_1(\bar B\baut_\partial(N_I))\cong \operatorname{Image}(\tilde J_1) \cong \operatorname{Image}(J_0)\subset \pi_0 (\aut_\partial (N_I))$$ on $f$ is given by the composition $$X \xrightarrow{f}N_I\times D^k\xrightarrow{ \phi\times \operatorname{id}_{D^k}}N_I\times D^k,$$ where $\phi$ is a diffeomorphism representing $[\phi]$ considered as an element of $\operatorname{Image}(J_0)$. \cite[Lemma 3.3.]{berg12} now implies that $$\eta((\phi\times \operatorname{id}_{D^k})\circ f)=((\phi\times \operatorname{id}_{D^k})^*)^{-1}(\eta(f))+\eta(\operatorname{id}_{D^k})=((\phi\times \operatorname{id}_{D^k})^*)^{-1}(\eta(f)),$$ using that the normal invariant of a diffeomorphism is trivial. This implies that $[\phi]$ acts on $\tilde H^*(N;\IQ)\otimes \pi_*(G/O)_k$ via $(\phi^{-1})^*\otimes \id_{\pi_*(G/O)}.$ But this exactly corresponds to the standard action under the isomorphism $$\tilde H^*(N_I;\IQ)\cong \hom_\IQ(\tilde H_*(N_I;\IQ);\IQ)\cong \tilde H_*(N_I;\IQ).$$ If $\phi$ lies in the kernel of the map $$\pi_1(\bar B\aut_\partial(N_I))\rightarrow \Gamma_I,$$ then it is in the kernel of $\tilde H_*$ and a similar argument as before show that it acts trivially on $\pi^{ab}_k(\mathcal{F}_I)$ (compare Lemma \ref{gamma I action on universal cover}). The compatibility with the stabilization maps follows from the fact that the isomorphisms (\ref{naturalmapandH^*}) is natural. 

The statement (2) follows from (1) by using the fact that $G/O$ and hence also the mapping-space $\map_*(N_I/\partial N_I,G/O)$ are infinite loop spaces. Thus all rational k-invariants vanish for $\map_*(N_I/\partial N_I,G/O).$ This is equivalent to: \\
For all elements $\alpha\in \pi_k(\map_*(N_I/\partial N_I,G/O))\otimes \IQ$ there exists an element $c\in H^k(\map_*(N_I/\partial N_I,G/O);\IQ),$ such that $c(h(\alpha))\neq 0,$ where $h$ denotes the rational Hurewicz homomorphism. Since $$\pi_k((\mathcal{F}_I)_{\IQ})\otimes \IQ\rightarrow \pi_k(\map_*(N_I/\partial N_I,G/O))\otimes \IQ$$ is injective, it follows that all rational k-invariants also vanish for $(\mathcal{F}_I)_{\IQ}$. This shows that $(\mathcal{F}_I)_{\IQ}$ is a product of Eilenberg-Maclane spaces and hence its homology is given by the free graded commutative algebra on its homotopy groups. Moreover the $\pi_1(\bar B\aut_\partial (N_I))$-action is induced by the standard action.
\end{proof}

We use the previous proposition to give a Schur multifunctor description of $H_r(\mathcal{F}_I;\IQ).$ For a multiindex $\mu$ with $\ell(\mu)=n-1,$ consider the $\Sigma_{\mu}$-modules
$\Pi(\mu)$ given by $$\Pi(0,...,1,...,0)=s^{-i}\pi_{*}(G/O)\otimes \IQ,$$ where the $1$ sits in the $i$-th position and $0$ otherwise. The corresponding Schur multifunctor $$\Pi:Mod(\IZ)^{n-1}\rightarrow Vect_*(\IQ),$$ has the property that there is an isomorphism of the induced $\Gamma_I$-modules
$$\Pi(H_I)\cong (\tilde H_{*}(N_I,\IQ) \otimes \pi_*(G/O))^+.$$

It follows now that we get an isomorphisms of $\Gamma_I$-modules $$\Lambda\circ \Pi(H_I)\cong \Lambda( (\tilde H_{*}(N_I,\IQ) \otimes \pi_*(G/O)^+))\cong H_*(\mathcal{F}_I,\IQ),$$ where the left-hand $\Lambda$ denotes the free graded commutative algebra endofunctor of $Vect_*(\IQ).$ Recall that $\Lambda$ is given as the Schur functor with $\Lambda(n)=\IQ[n]$ and trivial $\Sigma_n$-action. Now setting $\mathscr{H}_r=\Lambda_r\circ \Pi$ and observing that $\Lambda_r$ is of degree $\leq r$ and $\Pi$ additive, we get the following:

\begin{prop}\label{Fschur}
There is an isomorphism of $\Gamma_I$-modules
$$H_r(\mathcal{F}_I;\IQ)\cong \bigoplus_\mu \mathscr{H}_r(\mu) \otimes_\mu H_I^{\otimes_\eta},$$ compatible with the stabilization maps, where the $\mathscr{H}_r(\mu)$ are trivial for $|\mu|>r.$
\end{prop}

Now we proof the second main theorem of this paper.

\begin{THMB}
The stabilization map $$H_i(B\bdiff_\partial(N_I);\IQ)\rightarrow H_i(B\bdiff_\partial(N_{I'});\IQ)$$
 is an isomorphism for $g_p> 2i+2$ when $2p\neq n$ and $g_p> 2i+4$ id $2p=n$ and an epimorphism for $g_p\geq 2i+2$ respectively $g_p\geq 2i+4.$
\end{THMB}

\begin{proof}Denote by $Y_I=B\bdiff_\partial(N_I)$ and $\bar X_I=\bar B\baut_\partial(N_I).$
Consider the Serre spectral sequences of the homotopy fibration $$\mathcal{F}_I\rightarrow Y_I\rightarrow  \bar  X_I$$ and the analog for $I'.$ The stabilization map induces maps on the $E_2$-pages 
$$\sigma_*:H_k(\bar X_I;H_l(\mathcal{F}_I;\IQ))\rightarrow H_k(\bar X_{I'};H_l(\mathcal{F}_{I'};\IQ)).$$ The theorem follows upon showing that these are isomorphisms for $g_p> 2k+2l+2$ $(+4$ if $p=n/2)$ and epimorphisms for $g_p\geq2k+2l+2$ $(+4$ if $p=n/2).$
For this we consider the universal coving spectral sequence $$H_r(\pi_1(\bar X_I);H_s(\bar  X_I\univ;H_l(\mathcal{F}_I;\IQ)))\Rightarrow H_{r+s}(\bar X_I;H_l(\mathcal{F}_I,\IQ)).$$ The condition above would follow upon showing that maps induced by the stabilization map on the $E^2$-page are isomorphisms for $g_p> 2r+2s+2l+2$ $(+4$ if $p=n/2)$ and epimorphisms for $g_p\geq2r+2s+2l+2$ $(+4$ if $p=n/2).$ To show this we observe that there are isomorphism of $\Gamma_I$-modules compatible with the stabilization maps:
$$H_s(\bar  X_I\univ;H_l(\mathcal{F}_I;\IQ))\cong H_s(\bar  X_I\univ)\otimes H_l(\mathcal{F}_I;\IQ))\cong H^{CE}_s(\g_I)\otimes H_l(\mathcal{F}_I;\IQ)),$$ where $\Gamma_I$ acts on the $2$-nd and $3$-rd term diagonally. Note that $\bar X_I$ and $X_I$ have the same universal cover, which is moreover naturally homotopy equivalent to $B\aut_\partial(N_I)\univ.$ The stability for $$H_r(\pi_1(\bar  X_I);H^{CE}_s(\g_I)\otimes H_l(\mathcal{F}_I;\IQ)))$$ follows from stability for $$H_r(\pi_1(\bar X_I);C^{CE}_s(\g_I)\otimes H_l(\mathcal{F}_I;\IQ))),$$ exactly like in Proposition \ref{algebraic homological stability result} upon using the two hyperhomology spectral sequences and that $\bar \pi_I$ is rationally perfect (Lemma \ref{H1 triv pi0}). Hence we are left with showing that the stabilization maps 
$$H_r(\pi_1(\bar X_I);C^{CE}_s(\g_I)\otimes H_l(\mathcal{F}_I;\IQ)))\rightarrow H_r(\pi_1(\bar X_{I'});C^{CE}_s(\g_{I'})\otimes H_l(\mathcal{F}_{I'};\IQ)))$$ are isomorphisms for $g_p> 2r+2s+2l+2$ $(+4$ if $p=n/2)$ and epimorphisms for $g_p\geq2r+2s+2l+2$ $(+4$ if $p=n/2).$ The Lyndon spectral sequence reduces this to the corresponding statement for $$H_r(\Gamma_I;C^{CE}_s(\g_I)\otimes H_l(\mathcal{F}_I;\IQ)))\rightarrow H_r(\Gamma_{I'};C^{CE}_s(\g_{I'})\otimes H_l(\mathcal{F}_{I'};\IQ))).$$
Propositions \ref{Fschur} and \ref{der iso compatible with stab maps} give us isomorphisms of $\Gamma_I$-modules compatible with the stabilization map $C^{CE}_s(\g_I)\otimes H_l(\mathcal{F}_I;\IQ))\cong \mathscr{C}_s(H_I) \otimes \mathscr{H}_l(H_I).$ The functor $\mathscr{C}_s$ is polynomial of degree $\leq s/2$ and the functor $\mathscr{H}_l$ is polynomial of degree $\leq l.$ The tensor product (in the sense of Schur multifunctors) $\mathscr{C}_s\otimes \mathscr{H}_l$ is of degree $\leq s/2+l.$ By Proposition \ref{general algebraic homological stability result} the stabilization maps
$$H_r(\Gamma_I;\mathscr{C}_s\otimes \mathscr{H}_l(H_I))\rightarrow H_r(\Gamma_{I'};\mathscr{C}_s\otimes \mathscr{H}_l(H_{I'}))$$ are isomorphisms for $g_p> 2r+s/2+l+2$ $(+4$ if $p=n/2)$ and epimorphisms for $g_p\geq 2r+s/2+l+2$ $(+4$ if $p=n/2),$ which finishes the proof.
\end{proof}

\bibliographystyle{alpha}
\bibliography{bib}

\begin{thebibliography}{GRW16}

\bibitem[Bak81]{MR632404}
A.~Bak.
\newblock {\em {$K$}-theory of forms}, volume~98 of {\em Annals of Mathematics
  Studies}.
\newblock Princeton University Press, Princeton, N.J.; University of Tokyo
  Press, Tokyo, 1981.

\bibitem[Bau96]{MR1340168}
H.~J. Baues.
\newblock On the group of homotopy equivalences of a manifold.
\newblock {\em Trans. Amer. Math. Soc.}, 348(12):4737--4773, 1996.

\bibitem[BM13]{berg12}
A.~Berglund and I.~Madsen.
\newblock Homological stability of diffeomorphism groups.
\newblock {\em Pure Appl. Math. Q.}, 9(1):1--48, 2013.

\bibitem[BM17]{berg13}
A.~Berglund and I.~Madsen.
\newblock {Rational homotopy theory of automorphisms of manifolds}.
\newblock {\em preprint, arXiv:1401.4096v2}, 2017.

\bibitem[BMS67]{bassmilnorserre}
H.~Bass, J.~Milnor, and J.-P. Serre.
\newblock Solution of the congruence subgroup problem for {${\rm
  SL}_{n}\,(n\geq 3)$} and {${\rm Sp}_{2n}\,(n\geq 2)$}.
\newblock {\em Inst. Hautes \'Etudes Sci. Publ. Math.}, (33):59--137, 1967.

\bibitem[Cha87]{MR885099}
R.~Charney.
\newblock A generalization of a theorem of {V}ogtmann.
\newblock In {\em Proceedings of the {N}orthwestern conference on cohomology of
  groups ({E}vanston, {I}ll., 1985)}, volume~44, pages 107--125, 1987.

\bibitem[Dwy80]{dwyer}
W.~G. Dwyer.
\newblock Twisted homological stability.
\newblock {\em Annals of Mathematics}, 111(2):239--251, March 1980.

\bibitem[EML54]{eilenberg}
S.~Eilenberg and S.~Mac~Lane.
\newblock On the groups {$H(\Pi,n)$}. {II}. {M}ethods of computation.
\newblock {\em Ann. of Math. (2)}, 60:49--139, 1954.

\bibitem[GRW14]{galatiusrandalwilliams}
S.~Galatius and O.~Randal-Williams.
\newblock Stable moduli spaces of high-dimensional manifolds.
\newblock {\em Acta Math.}, 212(2):257--377, 2014.

\bibitem[GRW16]{galatiusrandalwilliamshomstab}
S.~Galatius and O.~Randal-Williams.
\newblock {Homological stability for moduli spaces of high dimensional
  manifolds. I}.
\newblock {\em preprint, arXiv:1403.2334v4}, 2016.

\bibitem[Kah69]{MR0245031}
P.~J. Kahn.
\newblock Self-equivalences of {$(n-1)$}-connected {$2n$}-manifolds.
\newblock {\em Math. Ann.}, 180:26--47, 1969.

\bibitem[KM63]{MR0148075}
M.~A. Kervaire and J.~W. Milnor.
\newblock Groups of homotopy spheres. {I}.
\newblock {\em Ann. of Math. (2)}, 77:504--537, 1963.

\bibitem[Kre79]{MR561244}
M.~Kreck.
\newblock Isotopy classes of diffeomorphisms of {$(k-1)$}-connected
  almost-parallelizable {$2k$}-manifolds.
\newblock In {\em Algebraic topology, {A}arhus 1978 ({P}roc. {S}ympos., {U}niv.
  {A}arhus, {A}arhus, 1978)}, volume 763 of {\em Lecture Notes in Math.}, pages
  643--663. Springer, Berlin, 1979.

\bibitem[KS77]{MR0645390}
R.~C. Kirby and L.~C. Siebenmann.
\newblock {\em Foundational essays on topological manifolds, smoothings, and
  triangulations}.
\newblock Princeton University Press, Princeton, N.J.; University of Tokyo
  Press, Tokyo, 1977.
\newblock With notes by John Milnor and Michael Atiyah, Annals of Mathematics
  Studies, No. 88.

\bibitem[Lev69]{MR0248848}
J.~Levine.
\newblock Self-equivalences of {$S^{n}\times S^{k}$}.
\newblock {\em Trans. Amer. Math. Soc.}, 143:523--543, 1969.

\bibitem[LS07]{lupton}
G.~Lupton and S.~B. Smith.
\newblock Rationalized evaluation subgroups of a map. {II}. {Q}uillen models
  and adjoint maps.
\newblock {\em J. Pure Appl. Algebra}, 209(1):173--188, 2007.

\bibitem[LV12]{algop}
J.-L. Loday and B.~Vallette.
\newblock {\em Algebraic operads}, volume 346 of {\em Grundlehren der
  Mathematischen Wissenschaften}.
\newblock Springer, Heidelberg, 2012.

\bibitem[Per16a]{perlmutter}
N.~Perlmutter.
\newblock {Homological stability for the moduli spaces of products of spheres.}
\newblock {\em {Trans. Am. Math. Soc.}}, 368(7):5197--5228, 2016.

\bibitem[Per16b]{perlmutter2}
N.~Perlmutter.
\newblock {Linking forms and stabilization of diffeomorphism groups of
  manifolds of dimension $4n + 1$.}
\newblock {\em {J. Topol.}}, 9(2):552--606, 2016.

\bibitem[Qui69]{MR0258031}
D.~Quillen.
\newblock Rational homotopy theory.
\newblock {\em Ann. of Math. (2)}, 90:205--295, 1969.

\bibitem[Qui70]{quinn}
F.~Quinn.
\newblock A geometric formulation of surgery.
\newblock In {\em Topology of {M}anifolds ({P}roc. {I}nst., {U}niv. of
  {G}eorgia, {A}thens, {G}a., 1969)}, pages 500--511. Markham, Chicago, Ill.,
  1970.

\bibitem[Sat69]{MR0253369}
H.~Sato.
\newblock Diffeomorphism group of {$S^{p}\times S^{q}$} and exotic spheres.
\newblock {\em Quart. J. Math. Oxford Ser. (2)}, 20:255--276, 1969.

\bibitem[Tur69]{MR0250323}
E.~C. Turner.
\newblock Diffeomorphisms of a product of spheres.
\newblock {\em Invent. Math.}, 8:69--82, 1969.

\bibitem[vdK80]{vdk}
W.~van~der Kallen.
\newblock Homology stability for linear groups.
\newblock {\em Invent. Math.}, 60(3):269--295, 1980.

\bibitem[Wal62]{MR0145540}
C.~T.~C. Wall.
\newblock Classification of {$(n-1)$}-connected {$2n$}-manifolds.
\newblock {\em Ann. of Math. (2)}, 75:163--189, 1962.

\bibitem[Wal63]{MR0156354}
C.~T.~C. Wall.
\newblock Classification problems in differential topology. {II}.
  {D}iffeomorphisms of handlebodies.
\newblock {\em Topology}, 2:263--272, 1963.

\bibitem[Whi50]{MR0041435}
G.~W. Whitehead.
\newblock A generalization of the {H}opf invariant.
\newblock {\em Ann. of Math. (2)}, 51:192--237, 1950.

\bibitem[WW01]{MR1818774}
M.~Weiss and B.~Williams.
\newblock Automorphisms of manifolds.
\newblock In {\em Surveys on surgery theory, {V}ol. 2}, volume 149 of {\em Ann.
  of Math. Stud.}, pages 165--220. Princeton Univ. Press, Princeton, NJ, 2001.

\end{thebibliography}

\end{document}